\documentclass[onefignum,onetabnum]{siamart220329}
\usepackage{cite}

\usepackage{lipsum}
\usepackage{amsfonts}
\usepackage{tikz-cd}  
\usepackage{amsmath,booktabs,ctable,threeparttable}
\usepackage{amssymb,amsfonts,boxedminipage}
\usepackage{graphicx,epstopdf} 

\usepackage[caption=false]{subfig} 
\usepackage{algorithm,algpseudocode}
\usepackage{qbordermatrix}
\usepackage{blkarray}
\usepackage{arydshln}
\usepackage{lineno}
\usepackage{enumitem}

\ifpdf
  \DeclareGraphicsExtensions{.eps,.pdf,.png,.jpg}
\else
  \DeclareGraphicsExtensions{.eps}
\fi

\algrenewcommand\algorithmicrequire{\textbf{Input:}}
\algrenewcommand\algorithmicensure{\textbf{Output:}}

\newcommand\argmin{\mathop{\mathrm{argmin}}}
\newcommand\calR{\mathcal{R}}
\newcommand\calN{\mathcal{N}}

\newcommand\calX{\mathcal{X}}
\newcommand\calY{\mathcal{Y}}
\newcommand\calD{\mathcal{D}}
\newcommand\calA{\mathcal{A}}

\newcommand\calP{\mathcal{P}}
\newcommand\calS{\mathcal{S}}

\newsiamremark{hypothesis}{Hypothesis}
\crefname{hypothesis}{Hypothesis}{Hypotheses}
\newsiamthm{claim}{Claim}

\numberwithin{equation}{section}
\numberwithin{figure}{section}
\numberwithin{table}{section}

\headers{New interpretation of weighted pseudoinverse}{H. Li}

\title{A new interpretation of the weighted pseudoinverse and its applications 
	\thanks{Submitted to the editors DATE.
}}

\author{Haibo Li\thanks{School of Mathematics and Statistics, The University of Melbourne, Parkville, VIC 3010, Australia.
  (\email{haibo.li@unimelb.edu.au}).}
}

\usepackage{amsopn}

\ifpdf
\hypersetup{
  pdftitle={A new interpretation of the weighted pseudoinverse and its applications 
  },
  pdfauthor={H. Li}
}
\fi

\begin{document}

\maketitle

\begin{abstract}
	Consider the generalized linear least squares (GLS) problem $\min\|Lx\|_2 \ \mathrm{s.t.} \ \|M(Ax-b)\|_2=\min$. The weighted pseudoinverse $A_{ML}^{\dag}$ is the matrix that maps $b$ to the minimum 2-norm solution of this GLS problem.
	By introducing a linear operator induced by $\{A, M, L\}$ between two finite-dimensional Hilbert spaces, we show that the minimum 2-norm solution of the GLS problem is equivalent to the minimum norm solution of a linear least squares problem involving this linear operator, and $A_{ML}^{\dag}$ can be expressed as the composition of the Moore-Penrose pseudoinverse of this linear operator and an orthogonal projector. With this new interpretation, we establish the generalized Moore-Penrose equations that completely characterize the weighted pseudoinverse, give a closed-form expression of the weighted pseudoinverse using the generalized singular value decomposition (GSVD), and propose a generalized LSQR (gLSQR) algorithm for iteratively solving the GLS problem. We construct several numerical examples to test the proposed iterative algorithm for solving GLS problems. Our results highlight the close connections between GLS, weighted pseudoinverse, GSVD and gLSQR, providing new tools for both analysis and computations.
\end{abstract}



\section{Introduction}
Consider the generalized linear least squares (GLS) problem 
\begin{equation}\label{GLS0}
	\min_{x\in\mathbb{R}^{n}}\|Lx\|_2 \ \ \ \mathrm{s.t.} \ \ \ 
	\|M(Ax-b)\|_{2}=\min
\end{equation}
where $A\in\mathbb{R}^{m\times n}$, $M\in\mathbb{R}^{q\times m}$ and $L\in\mathbb{R}^{p\times n}$. In some literature it is also called the weighted linear least squares problem. The GLS problem generalizes the standard least squares (LS) problem $\min\|x\|_2 \ \mathrm{s.t.} \ \|Ax-b\|_{2}=\min$ by incorporating weighting matrices $M$ and $L$, which introduces additional constraints and objectives tailored to specific data characteristics. For instance, $M$ might be used to increase
the relative importance of accurate measurements, while $L$ could adjust the regularization or constraint structure to improve stability or enforce certain properties in the solution \cite[\S 6.1]{Golub2013}. Such problems arise in a variety of practical applications, including scatter data approximation \cite{wendland2004scattered}, functional data analysis \cite{Ramsay2009}, ill-posed inverse problems \cite{engl1996regularization}, surface fitting problems \cite{weiss2002advanced} and many others. 

The GLS problem is relatively simple when $L$ has full column rank, where it must have a unique solution. The case that both $L$ and $M$ have full column rank has been extensively studied in earlier literature; see e.g. \cite{price1964matrix,ben2006generalized}. For general rectangular matrices $M$ and $L$, in \cite{mitra1974projections} the authors proposed the concept of projection under seminorms to study the existence and structure of the solutions of \cref{GLS0}. A well-known result is that the uniqueness of the solution of \cref{GLS0} is equivalent to $\calN(MA)\cap\calN(L)=\{\mathbf{0}\}$, where $\calN(\cdot)$ is the null space of a matrix. In the case of nonuniqueness, there is a unique minimum 2-norm solution of \cref{GLS0}. The GLS problem was then intensively studied in \cite{elden1982weighted}, where the author gave the expression of the general solutions. Specifically, the author demonstrated that the minimum 2-norm solution can be written as $A_{ML}^{\dag}b$, where $A_{ML}^{\dag}\in\mathbb{R}^{n\times m}$ is the so-called $M,L$ weighted pseudoinverse of $A$ that shares several properties analogous to the Moore-Penrose pseudoinverse. 

The weighted pseudoinverse is a generalization of the Moore-Penrose pseudoinverse of a single matrix. Since the concept of the pseudoinverse was first introduced by Moore \cite{moore1920reciprocal,Moore1935-MOOGA} and later by Penrose \cite{penrose1955generalized,penrose1956best}, it has received considerable attention and many applications \cite{golub1965calculating,milne1968oblique,marquardt1970generalized,campbell2009generalized}, and there have been several various generalizations of the Moore-Penrose pseudoinverse, such as the restricted pseudoinverse proposed for linear-constrained LS problems \cite{minamide1970restricted,hartung1979note,callon1987method}, the product generalized inverse \cite{cline1970extension}, and another type of weighted pseudoinverse (different from that in this paper) \cite{chipman1964least,ward1977limit,ward1971weighted}. Among these generalizations, the weighted pseudoinverse proposed in \cite{elden1982weighted} has attracted significant attention. For example, in \cite{hansen1995lanczos}, the authors proposed an algorithm for computing the GSVD of $\{A,L\}$, where at each iteration $L_{IA}^{\dag}z$ needs to be computed for some vector $z$; here $I$ is the identity matrix. Moreover, using $L_{IA}^{\dag}$, the general-form Tikhonov regularization problem $\min_{x\in\mathbb{R}^{n}}\{\|Ax-b\|_{2}^2+\lambda\|Lx\|_{2}^2\}$ can be transformed to a standard-form problem, which is much easier for analysis and computations; see e.g. \cite{Hansen1998,Hansen2010}. 

However, computing the weighted pseudoinverse and solving the GLS problem are both quite challenging. Existing methods primarily depend on matrix factorizations, which are effective only for small-scale problems. For the case that $\calN(MA)\cap\calN(L)=\{\mathbf{0}\}$ with $M=I_{m}$, the identity matrix of order $m$, the author in \cite{elden1982weighted} gave a closed-form expression of $A_{IL}^{\dag}$ using the generalized singular value decomposition (GSVD) of $\{A,L\}$, which can be used to compute the solution of \cref{GLS0}. Furthermore, he proposed a more efficient algorithm for computing $A_{IL}^{\dag}$ based on QR factorizations, which avoids the GSVD computation. For large-scale matrices, however, there is a lack of efficient methods for computing $A_{ML}^{\dag}$ or for iteratively computing $A_{ML}^{\dag}b$ for a given $b$. This may partly be due to an insufficient understanding of the properties of the weighted pseudoinverse. In contrast, the properties of the Moore-Penrose pseudoinverse $A^{\dag}$ are well-established, and a variety of computational methods are available for it. For example, there is a closed-form expression of $A^{\dag}$ by using the singular value decomposition (SVD) of $A$, and $A^{\dag}b$ is the minimum 2-norm solution of $\min_{x\in\mathbb{R}^{n}}\|Ax-b\|_2$, which can be computed efficiently by the iterative solver LSQR \cite{Paige1982}. It would be beneficial to gain new insights into the weighted pseudoinverse and to establish deeper analogies with the Moore-Penrose pseudoinverse. This could enable the development of efficient iterative methods for computing $A_{ML}^{\dag}$.

In this paper, we provide a new interpretation of the weighted pseudoinverse and use it to design an iterative algorithm for computing $A_{ML}^{\dag}b$. To achieve this, we first introduce a linear operator $\calA$ between two finite dimensional Hilbert spaces, with non-Euclidean inner products induced by the matrices $\{A,M,L\}$. Then we formulate an operator-type LS problem involving $\calA$, showing that its minimum norm solution coincides with the minimum 2-norm solution of \cref{GLS0}. This result establishes a connection between $A_{ML}^{\dag}$ and $\calA^{\dag}$, the Moore-Penrose pseudoinverse of $\calA$. Building on this connection, we derive a set of generalized Moore-Penrose equations that fully characterize the weighted pseudoinverse. Additionally, by using the generalized singular value decomposition (GSVD) of $\{A,L\}$, we give a closed-form expression for $A_{IL}^{\dag}$, which is applicable regardless of whether $\calN(A)\cap\calN(L)=\{\mathbf{0}\}$ or not. 
To address the practical computational challenges involving $A_{ML}^{\dag}$, we extend the classical Golub-Kahan bidiagonalization (GKB) method \cite{golub1965calculating} and propose a novel iterative algorithm called the \textit{generalized LSQR} (gLSQR). The design of this algorithm leverages the connection between $A_{ML}^{\dag}$ and $\calA^{\dag}$, which can efficiently compute $A_{ML}^{\dag}b$ by iteratively refining the solutions to the GLS problem \cref{GLS0} without requiring any matrix factorizations. To demonstrate the effectiveness of gLSQR, we construct several numerical examples of GLS problems and show its ability to compute solutions with high accuracy across these diverse scenarios. The results in this paper highlight the close connections between GLS, weighted pseudoinverse, GSVD and gLSQR, providing new tools for both analysis and computations of related applications.


The paper is organized as follows. In \Cref{sec2} we review several basic properties of the LS problem and Moore-Penrose pseudoinverse. In \Cref{sec3} we analyze the GLS problem from the perspective of an equivalent operator-type LS problem. Building on this perspective, we offer a new interpretation of the weighted pseudoinverse and present several basic properties. In \Cref{sec4} we generalize the GKB method and propose the gLSQR algorithm for iterative computing $A_{ML}^{\dag}b$. In \Cref{sec5} we construct several nontrial numerical examples to test the gLSQR algorithm. Finally, we conclude the paper in \Cref{sec6}.

Throughout the paper, we denote by $\calN(\cdot)$ and $\calR(\cdot)$ the null space and range space of a matrix or linear operator, respectively, denote by $\mathbf{0}$ the zero matrix/vector with orders clear from the context, and denote by $\mathrm{span}\{\cdot\}$ the subspace spanned by a group of vectors or columns of a matrix.

\section{Linear least squares and pseudoinverse of linear operators}\label{sec2}
We review several basic properties of the LS problems and the pseudoinverse of linear operators in the context of Hilbert spaces; see e.g. \cite{groetsch1977generalized,ben2006generalized} for more details. These properties will be used in the subsequent sections. 

Let $\calX$ and $\calY$ be two Hilbert spaces, and let $T:\calX\rightarrow\calY$ be a bounded linear operator where its adjoint is denoted by $T^{*}$. Consider the linear operator equation $Tx=y$, which has a solution if and only if $y\in\calR(T)$. Otherwise, we consider the least squares solution. An element $x\in\calX$ is called a least squares solution of $Tx=y$ if it satisfies
\begin{equation}
	\|Tx-y\|_{\calY} = \inf\{\|Tz-y\|_{\calY}: z\in\calX\} .
\end{equation}
If the set of all least squares solutions has an element of minimum $\calX$-norm, i.e.
\begin{equation}
	\|x\|_{\calX} = \inf\{\|z\|_{\calX}: z \ \mathrm{is \ a \ least \ squares \ solution \ of \ } Tx=y\},
\end{equation}
then we call such an $x$ a best-approximate solution of $Tx=y$. The following well-known result describes the existence and uniqueness of the least squares solution and best-approximate solution.

\begin{theorem}\label{LS_criterion}
	For the linear operator equation $Tx=y$, the following properties holds:
	\begin{enumerate}
		\item[(1)] it has a least squares solution if and only if $y\in\calR(T)+\calR(T)^{\perp}$;
		\item[(2)] if (1) is satisfied, then $x$ is a least squares solution if and only if 
			\begin{equation}\label{normal}
				T^{*}Tx = T^{*}y
			\end{equation}
			holds, which is called the normal equation;
		\item[(3)] if (1) is satisfied, then $x$ is the unique best-approximate solution if and only if \cref{normal} is satisfied and $x\in\calN(T)^{\perp}$.
	\end{enumerate}
\end{theorem}

The best-approximate solution is closely related to the Moore-Penrose pseudoinverse of $T$, which is the linear operator mapping $y$ to the best-approximate solution of $Tx=y$. Based on \Cref{LS_criterion}, the standard definition of the Moore-Penrose pseudoinverse is as follows.

\begin{definition}\label{moore_penrose}
	For the bounded linear operator $T:\calX\rightarrow\calY$, define its restriction as $\widetilde{T}:=T|_{\calN(T)^{\perp}}:\calN(T)^{\perp}\rightarrow\calR(T)$. The Moore-Penrose pseudoinverse $T^{\dag}$ of $T$ is defined as the unique linear extension of $\widetilde{T}^{-1}$ to $\calD(T^{\dag}):=\calR(T)+\calR(T)^{\perp}$ such that $\calN(T^{\dag})=\calR(T)^{\perp}$.
\end{definition}

It has been proved that the following four ``Moore-Penrose'' equations hold:
\begin{equation}\label{MP_equation}
	\begin{cases}
		TT^{\dag}T = T, \\
		T^{\dag}TT^{\dag} = T^{\dag}, \\
		T^{\dag}T = \calP_{\calN(T)^\perp}, \\
   		TT^{\dag} = \calP_{\overline{R(T)}}|_{\calD(T^{\dag})},
	\end{cases}
\end{equation}
where $\calP_{\calS}$ is the orthogonal operator onto a subspace $\calS$.
Moreover, the Moore-Penrose equations uniquely characterize $T^{\dag}$, i.e. there exists a unique linear operator $T^{\dag}$ that satisfies equations \cref{MP_equation}, and the last two conditions can even be relaxed as that $TT^{\dag}$ and $T^{\dag}T$ are two orthogonal projectors. The following limit property holds:
\begin{equation}\label{limit0}
	\lim_{\delta \searrow 0} \left(T^*T + \delta I\right)^{-1} T^*
	= \lim_{\delta \searrow 0} T^* \left(TT^* + \delta I\right)^{-1}
\end{equation}
where $I:\calX\rightarrow\calX$ is the identity operator. Using the pseudoinverse, the following well-known result describes the structure of the least squares solutions.
\begin{theorem}
	Let $y\in\calD(T^{\dag})$. Then $x^{\dag}:=T^{\dag}y$ is the unique best-approximate solution of $Tx=y$, and the set of all least squares solutions is $x^{\dag}+\calN(T)$.
\end{theorem}




Now we come back to the settings with matrices. By treating $A\in\mathbb{R}^{m\times n}$ as a linear operator between the Euclidean spaces $\mathbb{R}^{n}$ and $\mathbb{R}^{m}$, all the above results directly apply to $A$. Let the SVD of $A$ be $U^{\top}AV = \Sigma$,
where $\Sigma = \begin{pmatrix}
    \Sigma_{r} &  \\
			& \mathbf{0}
\end{pmatrix}\in\mathbb{R}^{m\times n}$ with $\Sigma_{r}=\mathrm{diag}(\sigma_1,\dots,\sigma_r)\in\mathbb{R}^{r\times r}$, $\sigma_1\geq\dots\geq\sigma_r>0$, and $U=(u_1,\dots,u_m)\in\mathbb{R}^{m\times m}$ and $V=(v_1,\dots,v_n)\in\mathbb{R}^{n\times n}$ are orthogonal matrices. The pseudoinverse of $A$ has the expression $A^{\dag}=V\Sigma^{\dag}U^{\top}$ with $\Sigma^{\dag} = \begin{pmatrix}
    \Sigma_{r}^{-1} &  \\
			& \mathbf{0}
\end{pmatrix}\in\mathbb{R}^{n\times m}$, and the LS problem $\min_{x\in\mathbb{R}^{n}}\|Ax-b\|_2$ has a unique minimum 2-norm solution $x^{\dag} = A^{\dag}b =\sum_{i=1}^{r}\frac{u_{i}^{\top}b}{\sigma_i}v_i $. We remark that for a compact linear operator $T$, there exists an analogous decomposition to SVD, called the singular value expansion (SVE). Using the SVE of $T$, we can also give a similar expression of $x^{\dag}$ for the linear operator equation; see e.g. \cite[\S 15.4]{Kress2014}.

For large-scale matrix $A$, the LSQR algorithm is an efficient iterative approach for the LS problems $\min_{x\in\mathbb{R}^{n}}\|Ax-b\|_2$. This algorithm is based on the GKB process, where the main computations are matrix-vector products involving $A$ and $A^{\top}$. At the $k$-th step, the GKB process of $A$ and $b$ generates two groups of 2-orthogonal vectors, which form orthonormal bases of the Krylov subspaces $\mathcal{K}_{k+1}(AA^{\top},b)$ and $\mathcal{K}_{k}(A^{\top}A,A^{\top}b)$, respectively. Meanwhile, it reduces $A$ to a $(k+1)\times k$ lower bidiagonal matrix, which is then used in the LSQR algorithm to iterative compute an approximate solution. Besides, the GKB process is often used as a precursor for computing a partial SVD of a large-scale matrix. 

To summarize this section, we remark that the LS problem, Moore-Penrose pseudoinverse, SVD, GKB process and LSQR algorithm are all closely related. Each of them plays an important role in matrix computation problems, from providing theoretical analysis tools to enhancing the efficiency of numerical computations. Clearer relationships among these concepts are illustrated in \Cref{fig0} at the end of \Cref{sec4}.


\section{Generalized linear least squares and weighted pseudoinverse}\label{sec3}
First, we present a result that characterizes the solutions of the GLS problem, which allows us to reformulate the GLS problem as an equivalent operator-type LS problem. We then provide a new interpretation of the weighted pseudoinverse and establish several of its basic properties.

\subsection{Generalized linear least squares}
In the following part, we use $\|x\|_{C}$ to denote the seminorm $(x^{\top}Cx)^{1/2}$ for a symmetric positive semidefinite $C$. It becomes a norm when $C$ is strictly positive definite. The following result provides a criterion for determining a solution of the GLS problem.
\begin{theorem}\label{thm:GLS_solv}
	For the GLS problem
	\begin{equation}\label{GLS}
		  \min_{x\in\mathbb{R}^{n}}\|x\|_{Q} \ \ \ \mathrm{s.t.} \ \ \
		  \|Ax-b\|_{P}=\mathrm{min} ,
	  \end{equation}
	  where $A\in\mathbb{R}^{m\times n}$, and $P\in\mathbb{R}^{m\times m}$, $Q\in\mathbb{R}^{n\times n}$ are symmetric positive semidefinite matrices, let $G=A^{\top}PA+Q$. The following properties hold:
	  \begin{enumerate}
		\item[(1)] if $x$ is a solution of \cref{GLS}, then $\calP_{\cal{R}(G)}x$ is also a solution; conversally, if $x\in\calR(G)$ is a solution, then $x+z$ is a solution for any $z\in\calN(G)$;
		\item[(2)] the vector $x\in\mathbb{R}^{n}$ is a solution of \cref{GLS} if and only if
		\begin{equation}\label{GLS_criterion}
		  \begin{cases}
			  A^{\top}P(Ax-b)=\mathbf{0} , \\
			  x^{\top}Gz = 0 , \ \ \ \forall \ z \in \mathcal{N}(A^{\top}PA) .
		  \end{cases}
		\end{equation}
	  \end{enumerate}
\end{theorem}
\begin{proof}
	First note that $G$ is symmetric positive semidefinite. Thus, any vector $x\in\mathbb{R}^{n}$ has the decomposition $x=\calP_{\calR(G)}x+\calP_{\calN(G)}x=:x_1+x_2$ and $x_1\perp x_2$, where $\perp$ is the orthogonal relation in Euclidean spaces. Using the relation $\calN(G)=\calN(A^{\top}PA)\cap \calN(Q)$, we can verify that 
	\begin{equation*}
		\|A(x_1+x_2)-b\|_{P} = \|Ax_1-b\|_{P}, \ \ \ 
		\|x_1+x_2\|_{Q} = \|x_1\|_{Q} .
	\end{equation*}
	The first property immediately follows.
	
	To prove the second property, suppose $x$ is a solution to \cref{GLS}. Then it is a solution to the problem $\min_{x}\|Ax-b\|_{P}$. Taking the gradient of it leads to $A^{\top}P(Ax-b)=\mathbf{0}$. Now $\calP_{\calR(G)}x$ is a solution to \cref{GLS}. Note that 
	\begin{equation}
		x^{\top}Gz = (\calP_{\calR(G)}x)^{\top}G(G^{\dag}Gz) = (\calP_{\calR(G)}x)^{\top}G(\calP_{\calR(G)}z) ,
	\end{equation}
	and $(\calR(G),\langle\cdot,\cdot \rangle_{G})$ is a Hilbert space with inner product $\langle x, x' \rangle_{G}:=x^{\top}Gx'$. We only need to prove $\calP_{\calR(G)}x \perp_{G}\calP_{\calR(G)}(\mathcal{N}(A^{\top}PA))$, where $\perp_{G}$ is the orthogonal relation in $(\calR(G),\langle\cdot,\cdot \rangle_{G})$. Since $\calP_{\calR(G)}(\mathcal{N}(A^{\top}PA))$ is a closed subspace of $(\calR(G),\langle\cdot,\cdot \rangle_{G})$, we have the decomposition $\calP_{\calR(G)}x=\bar{x}_1+\bar{x}_2$ such that $\bar{x}_{1}\in\calP_{\calR(G)}(\mathcal{N}(A^{\top}PA))$ and $\bar{x}_{2} \perp_{G} \calP_{\calR(G)}(\mathcal{N}(A^{\top}PA))$. We only need to prove $\bar{x}_1= \mathbf{0}$. First we prove $\calP_{\calR(G)}(\mathcal{N}(A^{\top}PA))\subseteq \mathcal{N}(A^{\top}PA)$. To see it, for any $z\in\mathbb{R}^{n}$ we use the decomposition $z=z_1+z_2$ such that $z_{1}\in\calN(G)$ and $z_{2}\in\calN(G)^{\perp}=\calR(G)$, which indicates $\calP_{\calR(G)}z=z_2$. Thus, if $z\in\mathcal{N}(A^{\top}PA)$, then $A^{\top}PAz_2=A^{\top}PAz-A^{\top}PAz_1=A^{\top}PAz=0$ since $z_{1}\in\calN(G)\subseteq \calN(A^{\top}PA)$.
	
	Notice that $\bar{x}_{1}\in \mathcal{N}(A^{\top}PA)$, which leads to $\|A(\bar{x}_1+\bar{x}_2)-b\|_{P} = \|A\bar{x}_2-b\|_{P}$. This indicates that $\bar{x}_2$ is a solution to $\min_{x}\|Ax-b\|_{P}$. Since 
	\[\bar{x}_{2}^{\top}Q\bar{x}_{1} =  \bar{x}_{2}^{\top}(A^{\top}PA+Q)\bar{x}_{1}=\bar{x}_{2}^{\top}G\bar{x}_{1}=0, \]
	we have
	\begin{equation*}
		\|\calP_{\calR(G)}x\|_{Q}^2 = \|\bar{x}_{1}\|_{Q}^2 + \bar{x}_{2}^{\top}Q\bar{x}_{1} + \|\bar{x}_{2}\|_{Q}^2
		= \|\bar{x}_{1}\|_{Q}^2 + \|\bar{x}_{2}\|_{Q}^2 \geq \|\bar{x}_{2}\|_{Q}^2 .
	\end{equation*}
	Specifically, if the second relation of \cref{GLS_criterion} is not satisfied, which means $\bar{x}_1\neq \mathbf{0}$, it must hold $\|\bar{x}_{1}\|_{Q}>0$, which can be proved as follows. If $\|\bar{x}_{1}\|_{Q}=0$, then $\bar{x}_1\in\calN(Q)$. Combining with $\bar{x}_1\in \calP_{\calR(G)}(\mathcal{N}(A^{\top}PA))\subseteq \mathcal{N}(A^{\top}PA)$, it must hold $\bar{x}_1\in\calN(Q)\cap\mathcal{N}(A^{\top}PA)=\calN(G)$. Using the relation $\bar{x}_1\in\calR(G)$ and $\calN(G)\cap \calR(G)=\{\mathbf{0}\}$, we obtain $\bar{x}_1=\mathbf{0}$. Therefore, if $\bar{x}_1\neq \mathbf{0}$, then $\|\bar{x}_{2}\|_{Q}<\|\calP_{\calR(G)}x\|_{Q}=\|x\|_{Q}$, contradicting with that $x$ is a solution. 

	Now we prove that \cref{GLS_criterion} is a sufficient condition. Since \cref{GLS} has at least one solution in $\calR(G)$, by the first property, we only need to prove that there is only one solution in $\calR(G)$ that satisfies \cref{GLS_criterion}. The existence has already been proved. To see the uniqueness, suppose $x_1$ and $x_2$ are such two solutions. Then is must hold that $x_1=x_2+z$ with $z\in\calN(A^{\top}PA)$. From $x_1, x_2\in\calR(G)$ and $x_1, x_2\perp_{G}\calP_{\calR(G)}(\mathcal{N}(A^{\top}PA))$ we obtain $z\in\calR(G)$ and $z\perp_{G}\calP_{\calR(M)}(\mathcal{N}(A^{\top}PA))$. Combining with $z\in\calN(A^{\top}PA)$ leads to $z\in\calP_{\calR(G)}(\mathcal{N}(A^{\top}PA))$. Therefore, we have $z\perp_{G}z$, leading to $z=\mathbf{0}$. This proves the uniqueness of the solution in $\calR(G)$ that satisfies \cref{GLS_criterion}.
\end{proof}

From \Cref{thm:GLS_solv} and its proof, we have the following result.
\begin{corollary}
	There exists a unique solution of \cref{GLS} in $\calR(G)$, which is the minimum 2-norm solution of \cref{GLS}. Denoting this solution by $x^{\dag}$, then the set of all the solutions of \cref{GLS} is $x^{\dag}+\calN(G)$.
\end{corollary}

Therefore, in order to solve \cref{GLS}, a key step is to seek the solution $x^{\dag}$ in $\calR(G)$. To investigate the property of $x^{\dag}$, we introduce the following linear operator:
\begin{equation}\label{Amap}
	\calA: (\calR(G),\langle\cdot,\cdot\rangle_{G}) \rightarrow (\calR(P),\langle\cdot,\cdot\rangle_{P}), \ \ \ 
	v \mapsto \calP_{\calR(P)}Av ,
\end{equation}
where $v$ and $Av$ are column vectors under the canonical bases of $\mathbb{R}^{n}$ and $\mathbb{R}^{m}$. It is obvious that $\calA$ is a bounded linear operator. Let 
\[\calA^{*}: (\calR(P),\langle\cdot,\cdot\rangle_{P}) \rightarrow (\calR(G),\langle\cdot,\cdot\rangle_{G}), \ \ \ 
u \mapsto \calA^{*}u \]
be the adjoint operator of $\calA$, which is defined by the relation $\langle \calA v, u \rangle_{P} = \langle \calA^{*} u, v \rangle_{G}$ for any $v\in(\calR(G),\langle\cdot,\cdot\rangle_{G})$ and $(\calR(P),\langle\cdot,\cdot\rangle_{P})$. The following result describes the effect of $\calA$ on a vector under the canonical bases.

\begin{lemma}\label{lem:adj}
	  Under the canonical bases of $\mathbb{R}^{n}$ and $\mathbb{R}^{m}$, for any $u\in(\calR(P),\langle\cdot,\cdot\rangle_{P})$, it holds
	  \begin{equation}
		\calA^{*}u = G^{\dag}A^{\top}Pu .
	  \end{equation}
\end{lemma}
\begin{proof}
	Under the canonical bases, it holds
	\begin{equation*}
		\langle \calA v, u \rangle_{P} = \langle \calA^{*} u, v \rangle_{G} \  \Leftrightarrow \ 
		(\calP_{\calR(P)}Av)^{\top}Pu = v^{\top}G(\calA^{*} u)  .
	\end{equation*}
	Since $v\in\calR(G)$, we have $\calP_{\calR(G)}v=v$ and 
	\[(\calP_{\calR(P)}Av)^{\top}P=(\calP_{\calR(P)}A\calP_{\calR(G)}v)^{\top}P
	= v^{\top}\calP_{\calR(G)}A^{\top}\calP_{\calR(P)}P = v^{\top}\calP_{\calR(G)}A^{\top}P. \]
	Thus, we have 
	\begin{equation*}
		v^{\top} \calP_{\calR(G)}A^{\top}Pu  = v^{\top}G(\calA^{*} u) \  \Leftrightarrow \  
		v^{\top} (\calP_{\calR(G)}A^{\top}Pu-G(\calA^{*} u)) =  0
	\end{equation*}
	for any $v\in\calR(G)$. Noticing that $\calP_{\calR(G)}A^{\top}Pu-G(\calA^{*} u)\in\calR(G)$, it follows that $\calP_{\calR(G)}A^{\top}Pu=G(\calA^{*} u)$. This equality can also be written as
	\begin{equation*}
		GG^{\dag}A^{\top}Pu=G(\calA^{*} u) \ \Rightarrow \ 
		G^{\dag}GG^{\dag}A^{\top}Pu=G^{\dag}G(\calA^{*} u) \  \Leftrightarrow  \ 
		G^{\dag}A^{\top}Pu = \calP_{\calR(G)}(\calA^{*} u) . 
	\end{equation*}
	Since $\calA^{*} u\in\calR(G)$, we immediately obtain $\calA^{*} u=G^{\dag}A^{\top}Pu$.
\end{proof}

Now we can reformulate the minimum 2-norm solution of \cref{GLS} as the solution of the following equivalent operator-type LS problem.

\begin{theorem}\label{thm:opLS_solv}
	Let $\calX:=(\calR(G),\langle\cdot,\cdot\rangle_{G})$ and $\calY:=(\calR(P),\langle\cdot,\cdot\rangle_{P})$.
	The minimum $\|\cdot\|_{\calX}$-norm solution of the LS problem
	\begin{equation}\label{operator_ls}
		\min_{v\in\calX} \|\calA v - \calP_{\calR(P)}b\|_{\calY}
	\end{equation}
	is the unique solution of \cref{GLS} in $\calR(G)$.
\end{theorem}
\begin{proof}
	First note that \cref{operator_ls} has a unique $\|\cdot\|_{\calX}$-norm solution. In fact, $v$ is such a solution if and only if 
	\begin{equation}\label{operator_ls_criterion}
		\begin{cases}
			\calA^{*}(\calA v-\calP_{\calR(P)}b)=\mathbf{0},  \\
			v 
			\perp_{\calX} \calN(\calA) , 
		\end{cases}
	\end{equation}
	where $\perp_{\calX}$ is the orthogonal relation in the Hilbert space $\calX$. We only need to prove the equivalence between \cref{GLS_criterion} and \cref{operator_ls_criterion} for $x\in\calR(G)$. For notational consistency, here we uniformly use $v$ instead of $x$. The proof includes the following two steps.

	Step 1: prove $A^{\top}P(Av-b)=\mathbf{0}\Leftrightarrow \calA^{*}(\calA v-\calP_{\calR(P)}b)=\mathbf{0}$ for any $v\in\calR(G)$. By \Cref{lem:adj}, we have 
	\begin{equation*}
		\calA^{*}(\calA v-\calP_{\calR(P)}b)=G^{\dag}A^{\top}P(\calP_{\calR(P)}Av-\calP_{\calR(P)}b) 
		= G^{\dag}A^{\top}P(Av-b) .
	\end{equation*}
	Thus, the ``$\Rightarrow$'' relation is obvious. To get the ``$\Leftarrow$'' relation, suppose $G^{\dag}A^{\top}P(Av-b):=G^{\dag}u=\mathbf{0}$. Let the square root decomposition of $P$ be $P=L_{P}^{\top}L_{P}$. Then 
	\begin{equation*}
		u\in\calN(G^{\dag})=\calN(G) = \calN(A^{\top}PA)\cap \calN(Q) = \calN(L_{P}A)\cap \calN(Q) \subseteq \calN(L_{P}A) ,
	\end{equation*}
	since $G$ is symmetric and $\calN(A^{\top}PA)=\calN(L_{P}A)$. On the other hand, we have $u=A^{\top}P(Av-b)\in\calR((L_{P}A)^{\top})=\calN(L_{P}A)^{\perp}$. Therefore, $u\in \calN(L_{P}A)\cap\calN(L_{P}A)^{\perp}=\{\mathbf{0}\}$, leading to $u=\mathbf{0}$. This proves the ``$\Leftarrow$'' relation.

	Step 2: prove $v\perp_{G}\calP_{\calR(G)}(\mathcal{N}(A^{\top}PA))\Leftrightarrow v \perp_{\calX} \calN(\calA)$ for any $v\in\calR(G)$. Since $v$ is a vector under the canonical basis, we only need to prove $\calP_{\calR(G)}(\mathcal{N}(A^{\top}PA))=\calN(\calA)$. Note that $\calN(\calA)=\{v\in\calR(G):\calP_{\calR(P)}Av=\mathbf{0}\}$. In the proof of \Cref{thm:GLS_solv} we have already shown that $\calP_{\calR(G)}(\mathcal{N}(A^{\top}PA))=\calR(G)\cap\mathcal{N}(A^{\top}PA)$. Thus, for any $v\in\calP_{\calR(G)}(\mathcal{N}(A^{\top}PA))$, we have 
	\begin{equation*}
		A^{\top}PAv=\mathbf{0}\ \Rightarrow \ 
		L_{P}Av = \mathbf{0} \ \Rightarrow  \ 
		\calP_{\calR(P)}Av = P^{\dag}L_{P}^{\top}L_{P}Av = \mathbf{0} .
	\end{equation*}
	This implies that $v\in\calN(\calA)$, and then $\calP_{\calR(G)}(\mathcal{N}(A^{\top}PA))\subseteq \calN(\calA)$. To prove the ``$\supseteq $'' relation, let $v\in\calN(\calA)$. Then $v\in\calR(G)$ and $Av\in\calN(P)$. It follows that $A^{\top}PAv=\mathbf{0}$, leading to $v\in\calR(G)\cap\mathcal{N}(A^{\top}PA)$. This proves $\calP_{\calR(G)}(\mathcal{N}(A^{\top}PA))\supseteq \calN(\calA)$. The whole proof is then completed.
\end{proof}

\Cref{thm:opLS_solv} enables us to study the GLS problem by applying the extensive tools available for LS problems, thereby facilitating both analysis and computations. Building on this theorem, we give a new interpretation of the weighted pseudoinverse in the following part.

\subsection{Weighted pseudoinverse and its properties}
In this part, we come back to the GLS problem of the form \cref{GLS0}. By setting $P=M^{\top}M$ and $Q=L^{\top}L$, the two problems \cref{GLS0} and \cref{GLS} are essentially the same. The following result proposed in \cite{elden1982weighted} gives the structure of a solution of the GLS problem.

\begin{theorem}[\cite{elden1982weighted}]\label{thm:elden}
	For any $A\in\mathbb{R}^{m\times n}$, $M\in\mathbb{R}^{q\times m}$ and $L\in\mathbb{R}^{p\times n}$, the problem
	\begin{equation}\label{GLS1}
		\min_{x\in\mathbb{R}^{n}}\|Lx\|_2 \ \ \ \mathrm{s.t.} \ \ \ 
		\|M(Ax-b)\|_{2}=\min
	\end{equation}
	has the general solution
	\begin{equation}
		x = (I_{n}-(L\calP_{\calN(MA)})^{\dag}L)(MA)^{\dag}Mb + \calP_{\calN(MA)}(I_{n}-(L\calP_{\calN(MA)})^{\dag}L\calP_{\calN(MA)})z,
	\end{equation}
	where $z\in\mathbb{R}^{n}$ is arbitrary. 
\end{theorem}

It is shown that $\calN(MA)\cap\calN(L)=\{\calP_{\calN(MA)}(I_{n}-(L\calP_{\calN(MA)})^{\dag}L\calP_{\calN(MA)})z: z\in\mathbb{R}^{n}\}$, and $x^{\dag}:=(I_{n}-(L\calP_{\calN(MA)})^{\dag}L)(MA)^{\dag}Mb$ is $2$-orthogonal to $\calN(MA)\cap\calN(L)$, thereby it is the minimum 2-norm solution of \cref{GLS1}. To describe the map that takes $b$ to $x^{\dag}$,
define the matrix 
\begin{equation}
	A_{ML}^{\dag}:= (I_{n}-(L\calP_{\calN(MA)})^{\dag}L)(MA)^{\dag}M,
\end{equation}
which is called the $ML$-weighted pseudoinverse of $A$. Note that $x^{\dag}=A_{ML}^{\dag}b$ and note from \Cref{thm:opLS_solv} that $x^{\dag}=\calA^{\dag}\calP_{\calR(P)}b$ for any $b\in\mathbb{R}^{m}$. We immediately have the following result.

\begin{theorem}\label{thm:AML}
	Following the notations in \Cref{thm:elden} and \Cref{thm:opLS_solv}, let $P=M^{\top}M$ and $Q=L^{\top}L$. Then under the canonical bases of $\mathbb{R}^{n}$ and $\mathbb{R}^{m}$, it holds
	\begin{equation}\label{op_wpi}
		A_{ML}^{\dag} = \calA^{\dag}\calP_{\calR(P)} . 
	\end{equation}
\end{theorem}

\Cref{thm:AML} establishes a connection between the weighted pseudoinverse and the Moore-Penrose pseudoinverse. Specifically, if $M$ has full column rank, then $\calP_{\calR(P)}=I_{m}$, and $A_{ML}^{\dag}$ is essentially the pseudoinverse $\calA^{\dag}$. Using this new interpretation, we derive the following ``generalized'' Moore-Penrose equations to characterize the weighted pseudoinverse.

\begin{theorem}
	For the weighted pseudoinverse $X:=A_{ML}^{\dag}$, the following generalized Moore-Penrose equations hold:
	\begin{equation}\label{GMP_equation}
		\begin{cases}
			XAX = X, \\
			MAXA=MA, \\
			(M^{\top}MAX)^{\top}=M^{\top}MAX, \\
			(GXAG^{\dag})^{\top}=XA, \\
			XM^{\dag}M=X .
		\end{cases}
	\end{equation}
	Moreover, $A_{ML}^{\dag}$ is the unique solution of the above matrix equations.
\end{theorem}
\begin{proof}
	Using \cref{op_wpi} and $\calP_{\calR(P)}=\calP_{\calN(P)^{\perp}}=\calP_{\calN(M)^{\perp}}=M^{\dag}M$, the fifth identity immediately follows. We use the identities \cref{MP_equation} to prove the first four identities. Since $\calA^{\dag}\calA\calA^{\dag}=\calA^{\dag}$, for any $u\in\mathbb{R}^{m}$ under the canonical basis, it holds
	\begin{equation*}
		A_{ML}^{\dag}u = \calA^{\dag}\calP_{\calR(P)}u = 
		\calA^{\dag}\calA\calA^{\dag}\calP_{\calR(P)}u = \calA^{\dag}\calP_{\calR(P)}AA_{ML}^{\dag}u
		= A_{ML}^{\dag}AA_{ML}^{\dag}u ,
	\end{equation*}
	which implies the first identity. Notice that $\calP_{\calR(P)}A\calP_{\calN(G)}v=P^{\dag}PA\calP_{\calN(G)}v=\mathbf{0}$ due to $\calN(G)\subseteq\calN(A^{\top}PA)$, implying that $\calP_{\calR(P)}Av=\calP_{\calR(P)}A\calP_{\calR(G)}v$ for any $v\in\mathbb{R}^{n}$. Therefore, it holds 
	\begin{align*}
		\calP_{\calR(P)}Av &=
		\calA \calP_{\calR(G)}v = \calA\calA^{\dag}\calA \calP_{\calR(G)}v \\
		&= \calP_{\calR(P)}A\calA^{\dag}\calP_{\calR(P)}A\calP_{\calR(G)}v 
		= \calP_{\calR(P)}AA_{ML}^{\dag}Av.
	\end{align*}
	This implies $M^{\dag}MA=M^{\dag}MAA_{ML}^{\dag}A$, leading to the second identity.
	For the third identity, use the relation 
	\begin{equation*}
		\langle \calA\calA^{\dag}\calP_{\calR(P)}u, \calP_{\calR(P)}u'\rangle_{P}
		= \langle \calP_{\calR(P)}u, (\calA\calA^{\dag})^{*}\calP_{\calR(P)}u'\rangle_{P}
	\end{equation*}
	for any $u,u'\in\mathbb{R}^{m}$, which is equivalent to
	\begin{align*}
		& \ \ \ \ \ (\calP_{\calR(P)}AA_{ML}^{\dag}u)^{\top}P\calP_{\calR(P)}u'
		= (\calP_{\calR(P)}u)^{\top}P(\calA\calA^{\dag})^{*}\calP_{\calR(P)}u' \\
		& \Leftrightarrow \ u^{\top}(AA_{ML}^{\dag})^{\top}Pu' = u^{\top}P(\calA\calA^{\dag})^{*}\calP_{\calR(P)}u' .
	\end{align*}
	It follows that $P(\calA\calA^{\dag})^{*}\calP_{\calR(P)}u'=(AA_{ML}^{\dag})^{\top}Pu'$ for any $u'\in\mathbb{R}^{m}$. Using the fourth identity of \cref{MP_equation}, which implies $(\calA\calA^{\dag})^{*}=\calA\calA^{\dag}$, we have
	\begin{equation*}
		PAA_{ML}^{\dag}u'=P\calA\calA^{\dag}\calP_{\calR(P)}u'
		= (AA_{ML}^{\dag})^{\top}Pu',
	\end{equation*}
	leading to $ (AA_{ML}^{\dag})^{\top}M^{\top}M=M^{\top}MAA_{ML}^{\dag}$, which is just the third identity. For the fourth identity, use the relation
	\begin{equation*}
		\langle \calA^{\dag}\calA\calP_{\calR(G)}v, \calP_{\calR(G)}v'\rangle_{G}
		= \langle \calP_{\calR(G)}v, (\calA^{\dag}\calA)^{*}\calP_{\calR(G)}v'\rangle_{G}
	\end{equation*}
	for any $v,v'\in\mathbb{R}^{n}$, which is equivalent to
	\begin{align*}
		& \ \ \ \ \ (A_{ML}^{\dag}A\calP_{\calR(G)}v)^{\top}G\calP_{\calR(G)}v'
		= (\calP_{\calR(G)}v)^{\top}G(\calA^{\dag}\calA)^{*}\calP_{\calR(G)}v' \\
		& \Leftrightarrow \ v^{\top}(A_{ML}^{\dag}A)^{\top}Gv' = v^{\top}G(\calA^{\dag}\calA)^{*}\calP_{\calR(G)}v' ,
	\end{align*}
	where we have used $\calP_{\calR(P)}A\calP_{\calR(G)}v=\calP_{\calR(P)}Av$. It follows that $G(\calA^{\dag}\calA)^{*}\calP_{\calR(G)}v'=(A_{ML}^{\dag}A)^{\top}Gv'$ for any $v'\in\mathbb{R}^{n}$. Using the third identity of \cref{MP_equation}, which implies $(\calA^{\dag}\calA)^{*}=\calA^{\dag}\calA$, we have
	\begin{equation*}
		GA_{ML}^{\dag}Av' = G\calA^{\dag}\calA\calP_{\calR(G)}v' = (A_{ML}^{\dag}A)^{\top}Gv',
	\end{equation*}
	leading to $A_{ML}^{\dag}Av'=G^{\dag}GA_{ML}^{\dag}Av'=G^{\dag}(A_{ML}^{\dag}A)^{\top}Gv'$. Thus, it holds $A_{ML}^{\dag}A=G^{\dag}(A_{ML}^{\dag}A)^{\top}G$, which is the fourth identity.

	To show that $X$ is unique, first we prove that if $X$ satisfies the first four identities then $X$ must satisfy
	\begin{equation}\label{GMP_six}
		\calP_{\calR(G)}X=X .
	\end{equation}
	Using the first and fourth identities, we get
	\begin{align*}
		G^{\dag}GX = G^{\dag}GXAX=G^{\dag}G(G^{\dag}A^{\top}X^{\top}G)X
		= G^{\dag}A^{\top}X^{\top}GX = XAX=X .
	\end{align*}
	Suppose $Y$ is another matrix satisfying equations \cref{GMP_equation}, then $Y$ also satisfies \cref{GMP_six}. 
	Note that the fifth identity implies $X=X\calP_{\calR(P)}=XP^{\dag}P$. 
	It follows that
	\begin{align*}
		X &= (XA)X=(GX\calP_{\calR(P)}AG^{\dag})^{\top}X=(\calP_{\calR(P)}AG^{\dag})^{\top}X^{\top}GX \\
		&\overset{(1)}{=} (\calP_{\calR(P)}A\calP_{\calR(G)}YAG^{\dag})^{\top}X^{\top}GX=(G^{\dag}A^{\top}Y^{\top}\calP_{\calR(G)})A^{\top}(X\calP_{\calR(P)})^{\top}GX \\
		&\overset{(2)}{=} YAG^{\dag}A^{\top}X^{\top}GX=YA(XA)X=YAX=YAYAX=(YP^{\dag}P)AYAX \\
		&= YP^{\dag}(PAY)^{\top}AX=YP^{\dag}Y^{\top}A^{\top}PAX=Y(X^{\top}A^{\top}PAYP^{\dag})^{\top} \\
		&= Y(PAXAYP^{\dag})^{\top}=Y(PAYP^{\dag})^{\top}\overset{(3)}{=}Y\calP_{\calR(P)}AY=YAY=Y,
	\end{align*} 
	where for ``$\overset{(1)}{=}$" we use \cref{GMP_six} and $\calP_{\calR(P)}A=\calP_{\calR(P)}AYA$ due to the second identity, for ``$\overset{(2)}{=}$" we use $G^{\dag}A^{\top}Y^{\top}\calP_{\calR(G)}=YAG^{\dag}$ due to the fourth identity, and for ``$\overset{(3)}{=}$" we use $(PAYP^{\dag})^{\top}=\calP_{\calR(P)}AY$ due to the third identity. This proves the uniqueness of $X$.
\end{proof}

In \cite{elden1982weighted}, the author presented four identities similar to the Moore-Penrose equations of the pseudoinverse. The first three of them are identical to the first three listed in \cref{GMP_equation} while the fourth identity is $(L^{\top}LXA)^{\top}=L^{\top}LXA$. However, it remains an open problem whether the weighted pseudoinverse is uniquely determined by these four equations. In contrast, equations \cref{GMP_equation} completely characterize the weighted pseudoinverse. The first four identities correspond to the Moore-Penrose equations of a matrix, while the last identity describes an additional constraint on $A_{ML}^{\dag}$ arising from $M$. Specifically, the fifth identity is trivial when $M$ has full column rank.  

The weighted pseudoinverse satisfies the following limit property.

\begin{theorem}
	Following the notations in \Cref{thm:elden} and \Cref{thm:AML}, let $G=A^{\top}PA+Q$. It holds that 
	\begin{equation}\label{limit1}
		\lim_{\delta \searrow 0} (A^{\top}PA+\delta G)^{\dag}A^{\top}P = A_{ML}^{\dag} .
	\end{equation}
\end{theorem}
\begin{proof}
	First we show that for any $\delta>0$, under the canonical bases of $\mathbb{R}^{n}$ and $\mathbb{R}^{m}$, it holds for any $u\in\calR(P)$ that
	\begin{equation}\label{equal1}
		(\calA^{*}\calA+\delta I)^{-1}\calA^{*}u = (A^{\top}PA+\delta G)^{\dag}A^{\top}Pu .
	\end{equation}
	Notice that $\calN(A^{\top}PA+\delta G)=\calN(A^{\top}PA)\cap\calN(G)=\calN(G)$, which leads to that $\calR((A^{\top}PA+\delta G)^{\dag})=\calR(A^{\top}PA+\delta G)=\calR(G)$. Thus, we have $(A^{\top}PA+\delta G)^{\dag}A^{\top}Pu\in(\calR(G), \langle\cdot,\cdot \rangle_{G})$, and
	we only need to prove $(\calA^{*}\calA+\delta I)(A^{\top}PP+\delta G)^{\dag}A^{\top}Pu=\calA^{*}u$. By \Cref{lem:adj}, we have
	\begin{align*}
		(\calA^{*}\calA+\delta I)(A^{\top}PA+\delta G)^{\dag}A^{\top}Pu
		&= (G^{\dag}A^{\top}P\calP_{\calR(P)}A+\delta I_n)(A^{\top}PA+\delta G)^{\dag}A^{\top}Pu \\
		&= (G^{\dag}A^{\top}PA+\delta I_n)(A^{\top}PA+\delta G)^{\dag}A^{\top}Pu =: w, 
	\end{align*}
	and $\calA^{*}u=G^{\dag}A^{\top}Pu$. Therefore, we only need to show that $w$ is the minimum 2-norm solution of $\min_{x}\|Gx-A^{\top}Pu\|_2$, which is equivalent to that $w\in\calN(G)=\calR(G)$ and $G^{\top}Gw=G^{\top}A^{\top}Pu$. Since $\calR(G^{\dag})=\calR(G)$ and $\calR((A^{\top}PA+\delta G)^{\dag})=\calR(G)$, it follows that $w\in\calR(G)$. Also, we have
	\begin{align*}
		G^{\top}Gw 
		&= GG(G^{\dag}A^{\top}PA+\delta I_n)(A^{\top}PA+\delta G)^{\dag}A^{\top}Pu \\
		&= G(A^{\top}PA+\delta G)(A^{\top}PA+\delta G)^{\dag}A^{\top}Pu \\
		&= G\calP_{\calR(A^{\top}PA+\delta G)}A^{\top}Pu \\
		&= G\calP_{\calR(G)}A^{\top}Pu = G^{\top}A^{\top}Pu .
	\end{align*}
	This proves \cref{equal1}. 

	By \cref{equal1}, for any $b\in\mathbb{R}^{m}$, it holds that
	\begin{equation*}
		(\calA^{*}\calA+\delta I)^{-1}\calA^{*}\calP_{\calR(P)}b = (A^{\top}PA+\delta G)^{\dag}A^{\top}P\calP_{\calR(P)}b
		= (A^{\top}PA+\delta G)^{\dag}A^{\top}Pb, 
	\end{equation*}
	which indicates that $(\calA^{*}\calA+\delta I)^{-1}\calA^{*}\calP_{\calR(P)}=(A^{\top}PA+\delta G)^{\dag}A^{\top}P$ under the canonical bases. Using \Cref{thm:AML} and \cref{limit0}, we obtain
	\begin{align*}
		\lim_{\delta \searrow 0} (A^{\top}PA+\delta G)^{\dag}A^{\top}P 
		= \left(\lim_{\delta \searrow 0} (\calA^{*}\calA+\delta I)^{-1}\calA^{*}\right)\calP_{\calR(P)}
		= \calA^{\dag}\calP_{\calR(P)} = A_{ML}^{\dag}  .
	\end{align*}
	This completes the proof.
\end{proof}

In \cite[Theorem 2.4]{elden1982weighted}, the author gave a similar limit property:
\begin{equation}\label{limit2}
	\lim_{\delta \searrow 0} (A^{\top}PA+\delta Q)^{\dag}A^{\top}P = A_{ML}^{\dag}, 
\end{equation}
but did not include a proof. Notice that $A^{\top}PA+\delta G=(1+\delta)\left(A^{\top}PA+\frac{\delta}{1+\delta}Q\right)$. Therefore, \cref{limit1} and \cref{limit2} are equivalent. 

In many scenarios, people are more interested in \cref{GLS1} with $M=I_{m}$. In this case, we have $A_{IL}=\calA^{\dag}$, which means that the $I,L$-weighted pseudoinverse of $A$ is nothing but the pseudoinverse of the linear operator $\calA$. Moreover, it has a direct relation with the GSVD of the matrix pair $\{A, L\}$. Let us review the GSVD proposed in \cite{Paige1981}.

\begin{theorem}[GSVD]\label{thm_gsvd}
	Let $A\in\mathbb{R}^{m\times n}$ and $L\in\mathbb{R}^{p\times n}$. There exist orthogonal matrices $U_{A}\in \mathbb{R}^{m\times m}$, $U_{L}\in \mathbb{R}^{p\times p}$ and invertible matrix $X\in\mathbb{R}^{n\times n}$, such that the GSVD of $\{A, L\}$ has the form:
	\begin{subequations}\label{GSVD1}
	\begin{equation}\label{gsvd1}
	  A = U_{A}\Sigma_AX^{-1} , \ \ \  L = U_{L}\Sigma_LX^{-1} ,
  \end{equation}
	  with
	  \begin{equation}\label{gsvd2}
		  \Sigma_A =
		  \bordermatrix*[()]{%
			  C_{A} & \mathbf{0} & m \cr
			  r &   n-r   \cr
		  } \ , \  \ \ 
		  \Sigma_L = \bordermatrix*[()]{%
			  S_{L} & \mathbf{0} & p \cr
			  r &   n-r   \cr
		  } 
	  \end{equation}
	  and
	  \begin{equation}\label{gsvd3}
		  C_{A} =
		  \bordermatrix*[()]{%
			  I_{q_1}  &  &  & q_1 \cr
			  &  C_{q_2}  &  & q_2 \cr
			  &  & \mathbf{0}  & m-q_1-q_2 \cr
			  q_1 & q_2 & q_3
		  } , \ \ \
		  S_{L} =
		  \bordermatrix*[()]{%
			  \mathbf{0}  &  &  & p-r+q_1 \cr
			  &  S_{q_2}  &  & q_2 \cr
			  &  & I_{q_3}  & q_3 \cr
			  q_1 & q_2 & q_3
		  },
	  \end{equation}	
	  \end{subequations}
	  where $q_1+q_2+q_3=r=\mathrm{rank}((A^{\top},L^{\top})^{\top})$ and $C_{A}^{\top}C_A+S_{L}^{\top}S_L=I_{r}$. 
\end{theorem}

In \cite{elden1982weighted}, the author shows that if $\calN(A)\cap\calN(L)=\{\mathbf{0}\}$, then $A_{IL}^{\dag}=X\Sigma_{A}^{\dag}U_{A}^{\top}$. In the following result, we give a similar expression for $A_{IL}^{\dag}$, which is applicable regardless of whether $\calN(A)\cap\calN(L)=\{\mathbf{0}\}$ or not.

\begin{theorem}\label{thm:wp_gsvd}
	For any two matrices $A\in\mathbb{R}^{m\times n}$ and $L\in\mathbb{R}^{p\times n}$ with the GSVD as described in \Cref{thm_gsvd}, let $G=A^{\top}A+L^{\top}L$. Then
	\begin{equation}\label{equa2}
		A_{IL}^{\dag} = \calP_{\calR(G)}X\Sigma_{A}^{\dag}U_{A}^{\top} .
	\end{equation}
\end{theorem}
\begin{proof} 
	First we prove that 
	$x:=X\Sigma_{A}^{\dag}U_{A}^{\top}b$ is a solution of \cref{GLS1}. We only need to verify that the two conditions in \cref{GLS_criterion} are satisfied. Since $AX=U_{A}\Sigma_{A}$, we have
	\begin{align*}
		A^{\top}(Ax-b)
		= X^{-\top}\Sigma_{A}^{\top}U_{A}^{\top}(U_{A}\Sigma_{A}\Sigma_{A}^{\dag}U_{A}^{\top}b-b) 
		= X^{-\top}(\Sigma_{A}^{\top}\Sigma_{A}\Sigma_{A}^{\dag}-\Sigma_{A}^{\top})U_{A}^{\top}b = \mathbf{0},
	\end{align*}
	since we can easily verify that $\Sigma_{A}^{\top}\Sigma_{A}\Sigma_{A}^{\dag}=\Sigma_{A}^{\top}$. 
	For the second condition in \cref{GLS_criterion}, if we partition $X$ as
	\begin{equation}\label{X_part}
		X =
		\bordermatrix*[()]{%
			X_{1} & X_{2} & X_{3} & X_{4} & n \cr
			q_1 & q_2 & q_3 & n-r \cr
		} \ ,
	\end{equation}
	we can verify that $\calN(A^{\top}A)=\calN(A)=\calR((X_3 \ X_4))$. On the other hand, it holds that $x\in\calR(X\Sigma_{A}^{\dag})$ and 
	\begin{equation}
		X\Sigma_{A}^{\dag} = \begin{pmatrix}
			X_1 & X_{2} & X_{3}
		\end{pmatrix}C_{A}^{\dag}
		= \begin{pmatrix}
			X_1 & X_{2}C_{q_2}^{-1} & \mathbf{0}
		\end{pmatrix},
	\end{equation}
	which means that $x\in\calR((X_1 \ X_2))$. Notice that 
	\begin{align*}
		G
		= X^{-\top}\left(\begin{pmatrix}
			C_{A}^{\top}C_{A} &  \\
			& \mathbf{0}
		\end{pmatrix} + \begin{pmatrix}
			S_{L}^{\top}S_{L} &  \\
			& \mathbf{0}
		\end{pmatrix}\right)X^{-1}
		= X^{-\top}\begin{pmatrix}
			I_{r} &  \\
			& \mathbf{0}
		\end{pmatrix}X^{-1},
	\end{align*}
	leading to $X^{\top}GX=\begin{pmatrix}
		I_{r} &  \\
		& \mathbf{0}
	\end{pmatrix}$. Thus, $\calR(X_4)=\calN(G)$ and the columns of $(X_1 \ X_2 \ X_3)$ are mutually $G$-orthonormal. It follows that $x^{\top}Gz=0$ for any $z\in\calN(A^{\top}A)=\calR((X_3 \ X_4))$. 

	Note that $\calP_{\calR(G)}x=\calP_{\calR(G)}X\Sigma_{A}^{\dag}U_{A}^{\top}b\in\calR(G)$, which is the minimum 2-norm solution of \cref{GLS1} for an arbitrary $b\in\mathbb{R}^{n}$. We immediately obtain \cref{equa2}.
\end{proof}

This theorem provides a direct computational approach for $A_{IL}^{\dag}$ using the GSVD. However, if the matrices are very large, computing the GSVD is extremely expensive. In this case, we need an iterative approach to approximate $A_{ML}^{\dag}b$ for any given $b$.

\section{Iterative method for computing weighted pseudoinverse}\label{sec4}
By \Cref{thm:elden}, computing $A_{ML}^{\dag}b$ is equivalent to computing the minimum 2-norm solution of the GLS problem \cref{GLS}. We aim to approximate the solution of the GLS problem through an iterative process. The starting point comes from \Cref{thm:opLS_solv}. To solve the LS problem \cref{operator_ls}, we apply the GKB process to the operator $\calA$ between the two Hilbert spaces $\calX$ and $\calY$; see \cite{caruso2019convergence} for the GKB for LS problems in Hilbert spaces. Starting from the initial vector $\calP_{\calR(P)}b$, the recursive relations of GKB can be expressed as follows:
\begin{equation}\label{GKB_op}
	\begin{cases}
		\beta_1 u_1 = \calP_{\calR(P)}b,  \\
		\alpha_{i}v_i = \calA^{*}u_i -\beta_i v_{i-1},  \\
		\beta_{i+1}u_{i+1} = \calA v_{i} - \alpha_i u_i,
	\end{cases}
\end{equation}
where $u_{i}\in \calY$ and $v_i\in\calX$, and $\alpha_i$ and $\beta_{i}$ are positive scalars such that $\|v_i\|_{\calX}=\|u_{i}\|_{\calY}=1$. Note that $v_{0}:=\mathbf{0}$ for the initial step. We remark that it is assumed that $\calP_{\calR(P)}b\neq\mathbf{0}$; otherwise, the case is trivial because $A_{ML}^{\dag}b=\mathbf{0}$. Using \Cref{lem:adj}, we present the matrix form of the above recursive relations:
\begin{equation}\label{GKB_1}
	\begin{cases}
		\beta_1 u_1 = PP^{\dag}b,  \\
		\alpha_{i}v_i = G^{\dag}A^{\top}Pu_i -\beta_i v_{i-1},  \\
		\beta_{i+1}u_{i+1} = PP^{\dag}A v_{i} - \alpha_i u_i.
	\end{cases}
\end{equation}
Note that if $G=I_n$ and $P=I_m$, then the above recursive relations correspond to the standard GKB process of the matrix $A$. We name the iterative process corresponding to \cref{GKB_1} the \textit{generalized Golub-Kahan bidiagonalization} (\textsf{gGKB}). Before giving the practical computation procedure, let us explore how to further reduce the computational cost. In fact, computations involving $P^{\dag}$ can be avoided, as demonstrated by the following result.

\begin{lemma}
	For the recursive relations
	\begin{equation}\label{GKB_2}
		\begin{cases}
			\beta_1 \tilde{u}_1 = b,  \\
			\alpha_{i}\tilde{v}_i = G^{\dag}A^{\top}P\tilde{u}_i -\beta_i \tilde{v}_{i-1},  \\
			\beta_{i+1}\tilde{u}_{i+1} = A \tilde{v}_{i} - \alpha_i \tilde{u}_i,
		\end{cases}
	\end{equation}
	where $\tilde{v}_0:=\mathbf{0}$, if $\alpha_i$ and $\beta_i$ are the same as those in \cref{GKB_1}, then $v_i=\tilde{v}_i$ and $u_i=PP^{\dag}\tilde{u}_i$.
\end{lemma}
\begin{proof}
	We prove it by mathematical induction. For $i=1$, we have $\beta_{1}PP^{\dag}\tilde{u}_1=PP^{\dag}b=\beta_1u_1$, implying $u_{1}=PP^{\dag}\tilde{u}_1$. Note that $\tilde{v}_0=v_0$, and $Pu_1=PPP^{\dag}\tilde{u}_1=P\tilde{u}_1$. It follows that $\alpha_1v_1=\alpha_1\tilde{v}_1$, meaning $v_1=\tilde{v}_1$. Now assume $v_i=\tilde{v}_i$ and $u_i=PP^{\dag}\tilde{u}_i$ for $i\geq 1$. Then 
	\[PP^{\dag}Av_i-\alpha_iu_i=PP^{\dag}A\tilde{v}_i-\alpha_i PP^{\dag}\tilde{u}_i 
	= PP^{\dag}(A\tilde{v}_i-\alpha_i\tilde{u}_i)=\beta_{i+1}PP^{\dag}\tilde{u}_{i+1} . \]
	Thus, we have $\beta_{i+1}u_{i+1}=\beta_{i+1}PP^{\dag}\tilde{u}_{i+1}$, meaning $u_{i+1}=PP^{\dag}\tilde{u}_{i+1}$. Similar to the proof for $i=1$, using $Pu_{i+1}=PPP^{\dag}\tilde{u}_{i+1}=P\tilde{u}_{i+1}$, we can also prove $v_{i+1}=\tilde{v}_{i+1}$. 
\end{proof}

Using the above result, we can simplify the computation if we only need to generate $v_{i}$ but not $u_i$. At the initial step, we have
\[\beta_{1}=\|PP^{\dag}b\|_{P}=[(PP^{\dag}b)^{\top}PPP^{\dag}b]^{1/2}= (b^{\top}Pb)^{1/2} .\]
At the $i$-th step, to compute $\beta_{i+1}$, let $r_i=Av_{i} - \alpha_i \tilde{u}_i$. Then we have $\beta_{i+1}u_{i+1}=PP^{\dag}r_i$. Thus, it follows that
\[\beta_{i+1}=\|PP^{\dag}r_i\|_{P}=[(PP^{\dag}r_i)^{\top}PPP^{\dag}r_i]^{1/2}= (r_{i}^{\top}Pr_i)^{1/2} . \]
Now we can give the whole iterative procedure of the \textsf{gGKB} process, as shown in \Cref{alg:gGKB}.

Note that at each iteration of \textsf{gGKB} we need to compute $G^{\dag}\bar{s}$. For large-scale matrices, it is generally impractical to obtain $G^{\dag}$ directly. If $G$ is sparse and positive definite, we can first apply the sparse Cholesky factorization to $G$ and then compute $G^{\dag}$. Otherwise, using the fact that $G^{\dag}$ is the minimum 2-norm solution of $\min_{s\in\mathbb{R}^{n}}\|Gs-\bar{s}\|_2$, we can apply the iterative solver LSQR to $\min_{s\in\mathbb{R}^{n}}\|Gs-\bar{s}\|_2$ to approximate $G^{\dag}\bar{s}$ \cite{Paige1982}. 

\begin{algorithm}[htb]
	\caption{Generalized Golub-Kahan bidiagonalization (\textsf{gGKB})}\label{alg:gGKB}
	\begin{algorithmic}[1]
		\Require $A\in\mathbb{R}^{m\times n}$, $P\in\mathbb{R}^{m\times m}$, $Q\in\mathbb{R}^{n\times n}$, $b\in\mathbb{R}^{m}$
		\State Form $G=A^{\top}PA+Q$
		\State Compute $\beta_1=(b^{\top}Pb)^{1/2}$, \ $\tilde{u}_1=b/\beta_1$
		\State Compute $\bar{s}=A^{\top}Pu_1$, \ $s=G^{\dag}\bar{s}$
		\State $\alpha_{1}=(s^{\top}Gs)^{1/2}$, \ $v_{1}=s/\alpha_{1}$  
		\For {$i=1,2,\dots,k,$}
		\State $r=Av_i-\alpha_i\tilde{u}_i$
		\State $\beta_{i+1}=(r^{\top}Pr)^{1/2}$, \ $\tilde{u}_{i+1}=r/\beta_{i+1}$
		\State $\bar{s}=A^{\top}P\tilde{u}_{i+1}$, \ $s=G^{\dag}\bar{s}-\beta_{i+1}v_i$
		\State $\alpha_{i+1}= (s^{\top}Gs)^{1/2}$, \ $v_{i+1} = s/\alpha_{i+1}$
		\EndFor
		\Ensure $\{\alpha_i, \beta_i\}_{i=1}^{k+1}$, \ $\{\tilde{u}_i, v_i\}_{i=1}^{k+1}$  \Comment{$u_{i}=PP^{\dag}\tilde{u}_i$}
	\end{algorithmic}	
\end{algorithm}

We remark that in \cite{li2024characterizing} the author generalized the GKB process for computing nontrivial GSVD components of $\{A, L\}$. Here the proposed \textsf{gGKB} process is used to iteratively solve the GLS problem and is more versatile, as it can handle cases where $P$ is noninvertible.

The following result describes the subspaces generated by \textsf{gGKB}.

\begin{proposition}\label{prop:gGKB1}
	The \textsf{gGKB} process generates vectors $v_i\in\calR(G)$ and $u_i\in\calR(P)$, and $\{v_i\}_{i=1}^{k}$ is a $G$-orthonormal basis of the Krylov subspace 
	\begin{equation}\label{krylov1}
		\mathcal{K}_k(G^{\dag}A^{\top}PA, G^{\dag}A^{\top}Pb) = \mathrm{span}\{(G^{\dag}A^{\top}PA)^{i}G^{\dag}A^{\top}Pb\}_{i=0}^{k-1},
	\end{equation}
	and $\{u_i\}_{i=1}^{k}$ is a $P$-orthonormal basis of the Krylov subspace 
	\begin{equation}\label{krylov2}
		\mathcal{K}_k(PP^{\dag}AG^{\dag}A^{\top}P, PP^{\dag}b) = PP^{\dag}\mathrm{span}\{(AG^{\dag}A^{\top}P)^{i}b \}_{i=0}^{k-1}.
	\end{equation}
\end{proposition}
\begin{proof}
	The proof is based on the property of GKB for linear compact operators. As demonstrated above, the \textsf{gGKB} of $A$ is essentially the GKB of $\calA$ between the two Hilbert spaces $\calX$ and $\calY$. Therefore, the generated vectors satisfy $v_i\in\calR(G)$ and $u_i\in\calR(P)$, and $\{v_i\}_{i=1}^{k}$ and $\{u_i\}_{i=1}^{k}$ are $G$- and $P$-orthonormal bases of the Krylov subspaces $\mathcal{K}_k(\calA^{*}\calA,\calA^{*}\calP_{\calR(P)}b)$ and $\mathcal{K}_k(\calA\calA^{*},\calP_{\calR(P)}b)$, respectively. By \Cref{lem:adj}, we have
	\begin{equation*}
		(\calA^{*}\calA)^{i}\calA^{*}\calP_{\calR(P)}b = (G^{\dag}A^{\top}PPP^{\dag}A)^{i}G^{\dag}A^{\top}PPP^{\dag}b
		= (G^{\dag}A^{\top}PA)^{i}G^{\dag}A^{\top}Pb ,
	\end{equation*}
	and
	\begin{equation*}
		(\calA\calA^{*})^{i}\calP_{\calR(P)}b = (PP^{\dag}AG^{\dag}A^{\top}P)^{i}PP^{\dag}b
		= PP^{\dag}(AG^{\dag}A^{\top}P)^{i}b .
	\end{equation*}
	The desired result immediately follows.
\end{proof}

Using \cref{GKB_2}, one can also verify that $\mathrm{span}\{\tilde{u}_i\}_{i=1}^{k}=\mathcal{K}_k(AG^{\dag}A^{\top}P,b)$. Since the dimensions of $\calX$ and $\calY$ are $r_G=\mathrm{rank}(G)$ and $r_P=\mathrm{rank}(P)$, respectively, by \Cref{prop:gGKB1} the \textsf{gGKB} process will eventually terminate in at most $\min\{r_G,r_P\}$ steps. Here ``terminate'' means that $\alpha_{i}$ or $\beta_i$ equals zero at the current step, thereby the Krylov subspaces can not expand any longer. The ``terminate step" can be defined as
\begin{equation}\label{gGKB_termi}
	k_t=\min\{k: \alpha_{k+1}\beta_{k+1}=0\}.
\end{equation}
Suppose \textsf{gGKB} does not terminate before the $k$-th iteration, i.e. $\alpha_{i}\beta_{i}\neq 0$ for $1\leq i \leq k$. Then the $k$-step \textsf{gGKB} process generates a $G$-orthonormal matrix $V_{k}=(v_1,\dots,v_{k})\in \mathbb{R}^{n\times k}$ and a $P$-orthonormal matrix $U_{k}=(u_1,\dots,u_{k})\in \mathbb{R}^{m\times k}$, which satisfy the relations
\begin{equation}{\label{eq:GKB_matForm}}
	\begin{cases}
		\beta_1U_{k+1}e_{1} = PP^{\dag}b,  \\
		PP^{\dag}AV_k = U_{k+1}B_k,  \\
		G^{\dag}A^{\top}PU_{k+1} = V_kB_{k}^{T}+\alpha_{k+1}v_{k+1}e_{k+1}^\top ,
	\end{cases}
\end{equation}
where $e_1$ and $e_{k+1}$ are the first and $(k+1)$-th columns of $I_{k+1}$, and 
\begin{equation}
	B_{k}
	=\begin{pmatrix}
		\alpha_{1} & & & \\
		\beta_{2} &\alpha_{2} & & \\
		&\beta_{3} &\ddots & \\
		& &\ddots &\alpha_{k} \\
		& & &\beta_{k+1}
		\end{pmatrix}\in  \mathbb{R}^{(k+1)\times k}
\end{equation}
has full column rank. Note that it may happens that $\beta_{k+1}=0$, which means that \textsf{gGKB} terminates just at the $k$-th step and $u_{k+1}=\mathbf{0}$. 

Based on \textsf{gGKB}, we can design an iterative approach for solving \cref{operator_ls}, which will also solve \cref{GLS}. Note that under the canonical bases, we can rewrite \cref{operator_ls} as 
\begin{equation}\label{matrix_gls}
	\min_{x\in\calR(G)}\|PP^{\dag}A-PP^{\dag}b\|_P
\end{equation}
From $k=1$ onwards, we seek an approximate solution to \cref{matrix_gls} in the subspace $\mathrm{span}\{V_{k}\}$. By letting $x=V_{k}y$ with $y\in\mathbb{R}^{k}$, we obtain from \cref{eq:GKB_matForm} that
\begin{align*}
  &  \min_{x\in\mathrm{span}\{V_{k}\}}\|PP^{\dag}Ax-PP^{\dag}b\|_P 
  = \min_{y\in\mathbb{R}^{k}}\|PP^{\dag}AV_{k}y-\beta_1U_{k+1}e_{1}\|_P \\
  &= \min_{y\in\mathbb{R}^{k}}\|U_{k+1}(B_{k}y-\beta_1 e_{1})\|_P 
  = \min_{y\in\mathbb{R}^{k}}\|B_{k}y-\beta_1 e_{1}\|_2 ,
\end{align*}
where we have used that $\{u_{i}\}$ are $P$-orthonormal. Note that 
$$\argmin_{y\in\mathbb{R}^{k}}\|B_{k}y-\beta_1 e_{1}\|_2=B_k^{\dag}\beta_1 e_{1}=:y_k$$
since $B_k$ has full column rank.
Therefore, at the $k$-th iteration, the iterative approximation to \cref{matrix_gls} is given by
\begin{equation}\label{x_k}
  x_k = V_{k}y_k=V_{k}B_k^{\dag}\beta_1 e_{1}.
\end{equation}

The above approach is very similar to the LSQR algorithm for the standard LS problem \cite{Paige1982}. Moreover, the bidiagonal structure of $B_k$ enables the design of a recursive procedure to update  $x_k$ step by step, without explicitly computing $B_k^{\dag}\beta_1 e_{1}$ at each iteration. This procedure is based on the Givens QR factorization of $B_k$; see \cite[Section 4.1]{Paige1982} for details. Note that $u_i$ is not required for computing $x_k$, so there is no need to compute $PP^{\dag}$ in \textsf{gGKB}. We summarize the iterative algorithm for solving \cref{matrix_gls} in \Cref{alg:alg2}, which is named the \textit{generalized LSQR} (\textsf{gLSQR}) algorithm.

\begin{algorithm}[htb]
	\caption{Generalized LSQR (\textsf{gLSQR})}\label{alg:alg2}
	\begin{algorithmic}[1]
		\Require $A\in\mathbb{R}^{m\times n}$, $P\in\mathbb{R}^{m\times m}$, $Q\in\mathbb{R}^{n\times n}$, $b\in\mathbb{R}^{m}$
		\State \textbf{(Initialization)}
		\State Compute $\beta_1\tilde{u}_1=b$, $\alpha_1v_1=G^{\dag}A^{\top}P\tilde{u}_1$
    	\State Set $x_0=\mathbf{0}$, $w_1=v_1$, $\bar{\phi}_{1}=\beta_1$, $\bar{\rho}_1=\alpha_1$
		\For {$i=1,2,\dots$ until convergence,}
    \State \textbf{(Applying the gGKB process)}
		\State $\beta_{i+1}\tilde{u}_{i+1}=Av_{i}-\alpha_i\tilde{u}_{i}$ 
		\State $\alpha_{i+1}v_{i+1}=G^{\dag}A^{\top}P\tilde{u}_{i+1}-\beta_{i+1}v_{i}$
    \State \textbf{(Applying the Givens QR factorization to $B_k$)}
		\State $\rho_{i}=(\bar{\rho}_{i}^{2}+\beta_{i+1}^{2})^{1/2}$
		\State $c_{i}=\bar{\rho}_{i}/\rho_{i}$
    \State $s_{i}=\beta_{i+1}/\rho_{i}$
		\State$\theta_{i+1}=s_{i}\alpha_{i+1}$ 
    \State $\bar{\rho}_{i+1}=-c_{i}\alpha_{i+1}$
		\State $\phi_{i}=c_{i}\bar{\phi}_{i}$ 
    \State $\bar{\phi}_{i+1}=s_{i}\bar{\phi}_{i}$
    \State \textbf{(Updating the solution)}
		\State $x_{i}=x_{i-1}+(\phi_{i}/\rho_{i})w_{i} $
		\State $w_{i+1}=v_{i+1}-(\theta_{i+1}/\rho_{i})w_{i}$
		\EndFor
	\Ensure Approximate minimum 2-norm solution of \cref{GLS}: $x_k$
	\end{algorithmic}
\end{algorithm}

As the iteration proceeds, the $k$-th solution $x_k$ gradually approximates the true solution of \cref{matrix_gls}, and consequently of \cref{GLS}. We now state this property precisely.

\begin{theorem}\label{thm:naive}
  Suppose the \textsf{gGKB} process terminates at step $k_t$. Then  $x_{k_t}$ obtained by \textsf{gLSQR} is the exact minimum $2$-norm solution of \cref{GLS}.
\end{theorem}
\begin{proof}
	Since $x_{k_t}\in\mathrm{span}\{V_k\}\subseteq\calR(G)$, by \Cref{thm:GLS_solv}, we only need to verify that $x_{k_t}$ satisfies the two conditions in \cref{GLS_criterion}.

	Step 1: prove $A^{\top}P(Ax_{k_t}-b)=\mathbf{0}$. By writing $x_{k_t}$ as $x_{k_t}=V_{k_t}y_{k_t}$, we obtain from \cref{eq:GKB_matForm} that
	\begin{equation*}
		P(Ax_{k_t}-b) = P\calP_{\calR(P)}(Ax_{k_t}-b)
		= PU_{k_t+1}(B_{k_t}y_t-\beta_1 e_{1}) .
	\end{equation*} 
	Using \cref{eq:GKB_matForm} again, we get
	\begin{align*}
		G^{\dag}A^{\top}P(Ax_{k_t}-b) 
		&= G^{\dag}A^{\top}PU_{k_t+1}(B_{k_t}y_{k_t}-\beta_{1}e_1) \\
		&= (V_{k_t}B_{k_t}^{\top}+\alpha_{k_t+1}v_{k_t+1}e_{k+1}^{\top})(B_{k_t}y_{k_t}-\beta_{1}e_1) \\
		&= V_{k_t}(B_{k_t}^{\top}B_{k_t}y_{k_t}- B_{k_t}^{\top}\beta_{1}e_{1}) + \alpha_{k_t+1}\beta_{k_t+1}v_{k_t+1}e_{k_t}^{\top}y_{k_t} \\
		&= \alpha_{k_t+1}\beta_{k_t+1}v_{k_t+1}e_{k_t}^{\top}y_{k_t} \\
		&= \mathbf{0},
	  \end{align*}
	  since $\alpha_{k_t+1}\beta_{k_t+1}=0$ and $B_{k_t}^{\top}B_{k_t}y_{k_t}=B_{k_t}^{\top}\beta_{1}e_{1}$ due to $y_{k_t}=\argmin_{y}\|B_{k_t}y-\beta_{1}e_1\|_2$. Note that $x_{k_t}\in\calR(G)$. Using the same approach as in the proof (Step 1) of \Cref{thm:opLS_solv}, we obtain $A^{\top}P(Ax_{k_t}-b)=\mathbf{0}$.

	Step 2: prove $x_{k_t}^{\top}Gz=0$ for any $z\in\calN(A^{\top}PA)$. By \Cref{prop:gGKB1}, we have $x_{k_t}\in\calR(V_{k_t})\subseteq\calR(G^{\dag}A^{\top}P)$. Let $x_{k_t}=G^{\dag}A^{\top}Pw$. Then for any $z\in\calN(A^{\top}PA)$ we have
	\begin{align*}
		x_{k_t}^{\top}G z= (G^{\dag}A^{\top}Pw)^{\top}Gz
		= w^{\top}PAG^{\dag}Gz .
	\end{align*}
	Recall that $\calP_{\calR(G)}(\mathcal{N}(A^{\top}PA))\subseteq \mathcal{N}(A^{\top}PA)=\calN(L_{P}A)$, which has been proved in the proof  of \Cref{thm:GLS_solv}. Thus, $PAG^{\dag}Gz=L_{p}^{\top}L_{P}A\calP_{\calR(G)}z=\mathbf{0}$, leading to $x_{k_t}^{\top}Gz=\mathbf{0}$.
\end{proof}


An iterative algorithm should include a stopping criterion to decide whether the current iteration can be stopped to obtain a solution with acceptable approximation accuracy. Notice from \Cref{thm:opLS_solv} that $\calA^{*}(\calA x_{k}-\calP_{\calR(P)}b)=:\calA^{*}r_k$ would be zero at the iteration where the accurate solution is computed. The scaling invariant quantity $\|\calA^{*}r_{k}\|_G/(\|\calA\|\|\calP_{\calR(P)}b\|_P)$ can be used to measure the accuracy of the iterative solution, where $\|\calA\|$ is the operator norm defined as $\|\calA\|:=\max_{v\in\calR(G) \atop v\neq\mathbf{0}}\frac{\|\calA v\|_{P}}{\|v\|_{G}}$. Thus, we can use 
\begin{equation}\label{stop}
	\frac{\|\calA^{*}r_{k}\|_G}{\|\calA\|\|\calP_{\calR(P)}b\|_P} \leq \mathtt{tol}
\end{equation}
as a stopping criterion for \textsf{gLSQR}. To ensure computational practicality, we discuss how to compute the quantities in \cref{stop}. It is obviously that $\|\calP_{\calR(P)}b\|_P=(b^{\top}Pb)^{1/2}$. Using \cref{lem:adj} and the procedure in the proof of \Cref{thm:naive}, we get
\begin{align*}
	\|\calA^{*}r_{k}\|_G = \|G^{\dag}A^{\top}P\calP_{\calR(P)}(Ax_k-b)\|_G
	= \|\alpha_{k+1}\beta_{k+1}v_{k+1}e_{k}^{\top}y_{k}\|_G
	= \alpha_{k+1}\beta_{k+1}|e_{k}^{\top}y_{k}| ,
\end{align*}
where we have used $\|v_i\|_G=1$. Therefore, $\|\calA^{*}r_{k}\|_G$ can be computed quickly with very little additional cost.

To obtain an accurate estimate of $\|\calA\|$, here we consider the GLS problem \cref{GLS1} with $M=I_m$ for simplicity. Note that any GLS problem can be reduced to this form by substituting $A\leftarrow MA$ and $b\leftarrow Mb$. We give a matrix expression for $\|\calA\|$ using the GSVD of $\{A,L\}$. 

\begin{proposition}
	Suppose the GSVD of $\{A,L\}$ has the form \cref{GSVD1}. Then we have
	\begin{equation}
		\|\calA\| = \sigma_{\max}(C_{A}),
	\end{equation}
	which is the largest singular value of $C_{A}$.
\end{proposition}
\begin{proof}
	Using the expression of $\calA$ under the canonical bases, we have
	\begin{equation*}
		\|\calA\|=\max_{v\in\calR(G) \atop v\neq\mathbf{0}}\frac{\|\calA v\|_{2}}{\|v\|_{G}}
		= \max_{v\in\calR(G) \atop v\neq\mathbf{0}}\frac{v^{\top}A^{\top}Av}{v^{\top}Gv} .
	\end{equation*} 
	In the GSVD of $\{A,L\}$, we use the partition of $X$ as described in \cref{X_part} and denote $\widetilde{X}_1=(X_1 \ X_2 \ X_3)$. In the proof of \Cref{thm:wp_gsvd} we have shown that $\calN(G)=\calR(X_4)$. Now we show $\calR(G)=\calP_{\calR(G)}\calR(\widetilde{X}_1)$. We only need to show $\mathrm{dim}(\calP_{\calR(G)}\calR(\widetilde{X}_1))=r=\mathrm{rank}(G)$, which is equivalent to that $\calP_{\calR(G)}\widetilde{X}_1$ has full column rank. Suppose $\calP_{\calR(G)}\widetilde{X}_1w=\mathbf{0}$ for a $w\in\mathbb{R}^{r}$. Then $z:=\widetilde{X}_1w\in\calN(G)$. Thus, we have $0=z^{\top}Gz=w^{\top}\widetilde{X}_{1}^{\top}G\widetilde{X}_{1}z=z^{\top}z$, leading to $z=0$. This proves that $\calP_{\calR(G)}\widetilde{X}_1$ has full column rank.

	For any $v\in\calR(G)$, write $v=\calP_{\calR(G)}\widetilde{X}_1y$ with $y\in\mathbb{R}^{r}$. Then we have 
	\[v^{\top}Gv=y^{\top}\widetilde{X}_{1}^{\top}GG^{\dag}GGG^{\dag}\widetilde{X}_{1}y
	= y^{\top}\widetilde{X}_{1}^{\top}G\widetilde{X}_{1}y=\|y\|_2,\]
	and 
	\begin{equation*}
		Av = A\calP_{\calR(G)}\widetilde{X}_1y = A\widetilde{X}_1y - A\calP_{\calN(G)}\widetilde{X}_1y 
		= A\widetilde{X}_1y = U_{A}C_{A}y ,
	\end{equation*}
	where we have used $\calN(G)\subseteq\calN(A)$. It follows that $v^{\top}A^{\top}Av=\|U_{A}C_{A}y\|_2=\|C_{A}y\|_2$. Therefore, we obtain 
	\[\|\calA\| = \max_{y\in\mathbb{R}^{r} \atop y\neq\mathbf{0}}\frac{\|C_{A}y\|_2}{\|y\|_2}
	=\|C_{A}\|_2 = \sigma_{\max}(C_{A}) . \] 
	The proof is completed.
\end{proof}

Notice that $C_{A}$ is a diagonal matrix (not necessarily square). Thus, $\sigma_{\max}(C_{A})$ is the maximum value of the diagonals of $C_{A}$. For a regular matrix pair $\{A, L\}$, i.e. $G$ is nonsingular, several iterative algorithms exist that can rapidly compute the largest generalized singular values of $\{A, L\}$ \cite{Zha1996,jia2023joint}, thereby providing an accurate estimate of $\|\calA\|$. For a nonregular $\{A, L\}$, the method proposed in \cite{li2024characterizing} can accomplish the same task.

At the end of this section, we present a diagram to summarize the main ideas and findings of this paper. It illustrates how various concepts related to the LS problem and pseudoinverse have been generalized. Our results reveal the close connections between the GLS problem, weighted pseudoinverse, GSVD, and gLSQR. These insights improve theoretical understanding and offer tools for developing more effective computational methods for related applications.

\begin{figure}[!htbp]
	\centering
	\subfloat
	{\includegraphics[width=1.0\textwidth]{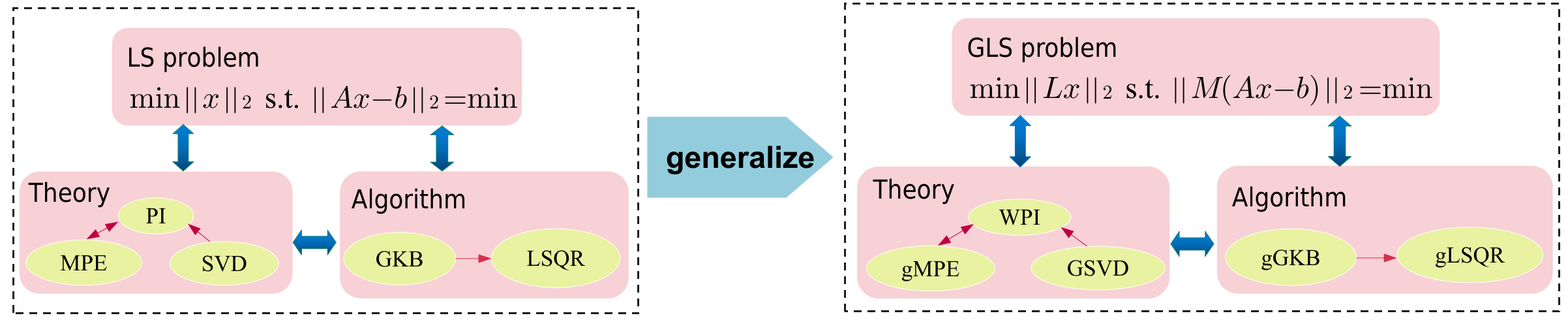}}\hspace{-0.0mm}
	\vspace{-4mm}
	\caption{The generalization of several concepts is illustrated as follows: PI and MPE stand for the Moore-Penrose pseudoinverse and Moore-Penrose equations, respectively. WPI and gMPE refer to the weighted pseudoinverse and generalized Moore-Penrose equations, respectively.}
	\label{fig0}
\end{figure}

\section{Numerical experiments}\label{sec5}
We use several numerical examples to demonstrate the performance of \textsf{gLSQR} for solving GLS problems. All the experiments are performed in MATLAB R2023b using double precision. We note that much of the existing literature on GLS problems is lacking in numerical results, partly because constructing nontrial test problems, especially for large-scale matrices, is challenging. Based on \Cref{thm:GLS_solv}, we construct a test GLS problem using the following steps:
\begin{enumerate}
	\item[(1)] Choose two matrices $A\in\mathbb{R}^{m\times n}$ and $L\in\mathbb{R}^{p\times n}$, where $m\leq n$. Compute $G=A^{\top}A+L^{\top}L$.
	\item[(2)] Construct a vector $w\in\mathcal{R}(G)$. Compute a matrix $B$ with columns that form a basis for $\calN(A)$. The true solution is constructed as 
	\begin{equation}\label{x_true}
		x^{\dag} = w - B(B^{\top}GB)^{-1}B^{\top}Gw .
	\end{equation}
	\item[(3)] Choose a vector $z\in\calR(A)^{\perp}$. Let the right-hand side vector be $b=Ax^{\dag}+z$.
\end{enumerate}
Note that \cref{x_true} ensures that $x^{\dag}$ satisfies \cref{GLS_criterion}. According to \Cref{thm:GLS_solv}, $x^{\dag}$ is the unique minimum 2-norm solution of \cref{GLS1} with $M=I_{m}$. For large-scale matrices, computing \cref{x_true} can be very challenging. Therefore, in our experiments, we only test small and medium-sized problems.

\paragraph*{Experiment 1}
The matrix $A\in\mathbb{R}^{2324\times 4486}$, named {\sf lp\_bnl2}, comes from linear programming problems and is sourced from the SuiteSparse Matrix Collection \cite{Davis2011}. The matrix $L=L_1$ is defined as the scaled discretization of the first-order differential operator:
\begin{equation*}
	L_1=\begin{pmatrix}
		1 & -1 &  &  \\
		&  \ddots & \ddots &  \\
		&  &  1  & -1 \\
	\end{pmatrix} \in \mathbb{R}^{(n-1)\times n}.
\end{equation*}
In this setup, $G$ is positive definite. We construct $w\in\mathbb{R}^{n}$ by evaluating the function $f(t)=t$ on a uniform grid over the interval $[0,1]$. To obtain the vector $z$, we compute the projection of the random vector \texttt{randn(m,1)} onto $\calR(A)^{\perp}$. In this experiment, we directly compute $G^{\dag}\bar{s}$ at each iteration of \textsf{gGKB} to simulate the exact computation.

\begin{figure}[!htbp]
	\centering
	\subfloat[]
	{\label{fig:1a}\includegraphics[width=0.45\textwidth]{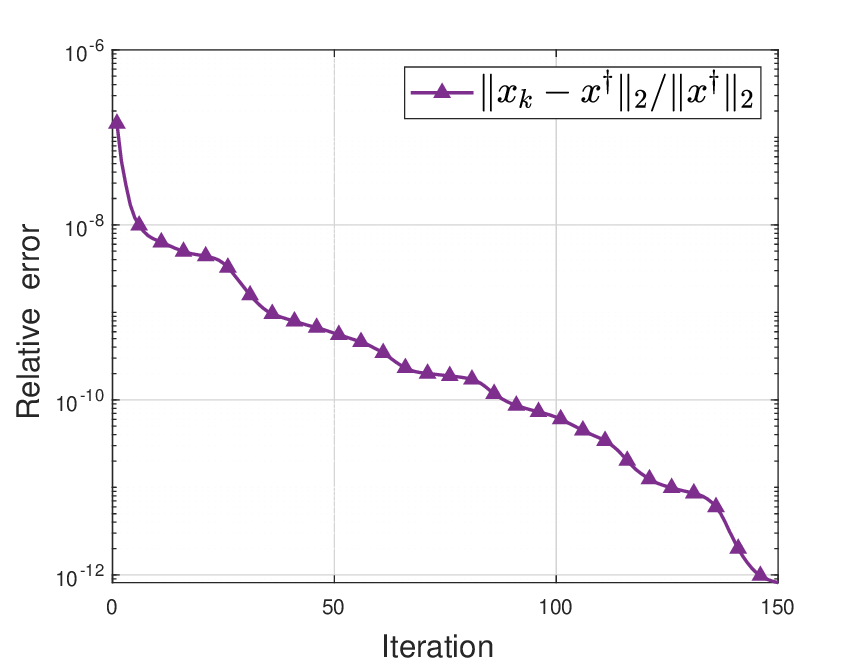}}\hspace{-2.5mm}
	\subfloat[]
	{\label{fig:1b}\includegraphics[width=0.48\textwidth]{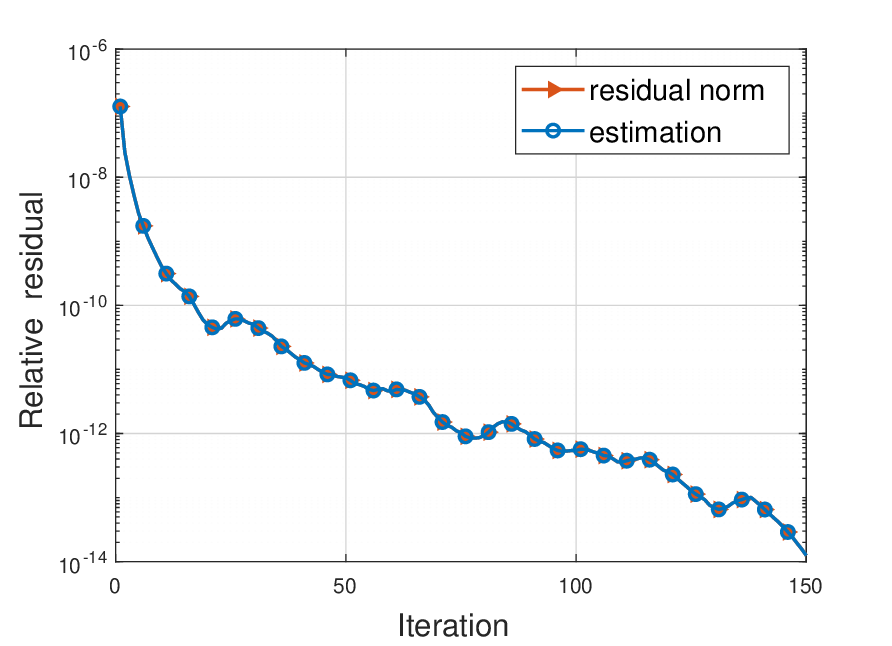}} \vspace{-0.0mm}
	\vspace{-2mm}
	\caption{Convergence history of \textsf{gLSQR} for the GLS problem with matrices \{{\sf lp\_bnl2}, $L_1$\}: (a) relative error of iterative solutions; (b) directly computed relative residual norm and its estimation.}
	\label{fig1}
\end{figure}

\begin{figure}[!htbp]
	\centering
	\subfloat
	{\label{fig:5c}\includegraphics[width=0.45\textwidth]{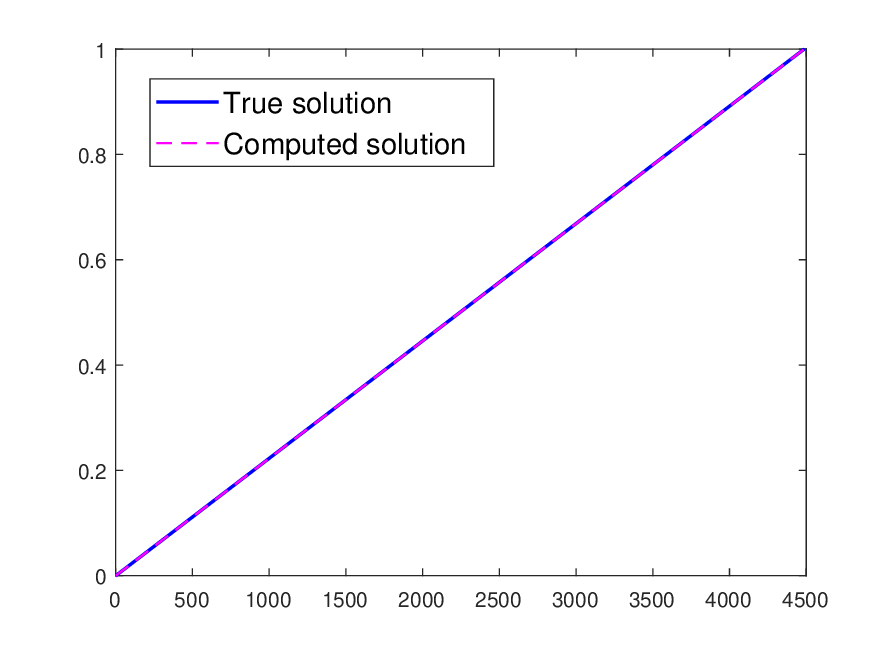}}\hspace{-0.0mm}
	\vspace{-3mm}
	\caption{True and computed solutions of the GLS problem with matrices \{{\sf lp\_bnl2}, $L_1$\}.}
	\label{fig2}
\end{figure}

The computational results obtained by \textsf{gLSQR} are displayed in \Cref{fig1} and \Cref{fig2}. For the residual norm, we use the directly computed quantity, i.e., the left-hand side of \cref{stop} as the true value, and $\alpha_{k+1}\beta_{k+1}|e_{k}^{\top}y_{k}|/(\sigma_{\max}(C_{A})\|b\|_2)$ for its estimation. These two quantities should be the same if all computations are performed accurately. This can be observed from \cref{fig:1b}, where the value gradually decreases at a very low level. The relative error curve shows that $x^{\dag}$ is gradually approximated by $x_{k}$. We plot the curve corresponding to $x_k$ at the final iteration alongside $x^{\dag}$, which shows that the two solutions match very closely.

\paragraph*{Experiment 2}
The matrix $A\in\mathbb{R}^{6334\times 7742}$ named {\sf TF15} arising from linear programming problems is taken from \cite{Davis2011}. The matrix $L=L_2$ is defined as the scaled discretization of the second-order differential operator:
\begin{equation*}
	L_2=\begin{pmatrix}
		-1 & 2 & -1 &  &  \\
		&  \ddots & \ddots & \ddots &  \\
		&  & -1 & 2 &  & -1 \\
	\end{pmatrix} \in \mathbb{R}^{(n-2)\times n}.
\end{equation*}
In this set, $G$ is positive definite. We construct $w\in\mathbb{R}^{n}$ by evaluating the function $f(t)=t^{3}-t^{2}$ on a uniform grid over the interval $[-1,1]$. In this experiment, we directly compute $G^{\dag}\bar{s}$ at each iteration of \textsf{gGKB} to simulate the exact computation.

\begin{figure}[!htbp]
	\centering
	\subfloat[]
	{\label{fig:3a}\includegraphics[width=0.45\textwidth]{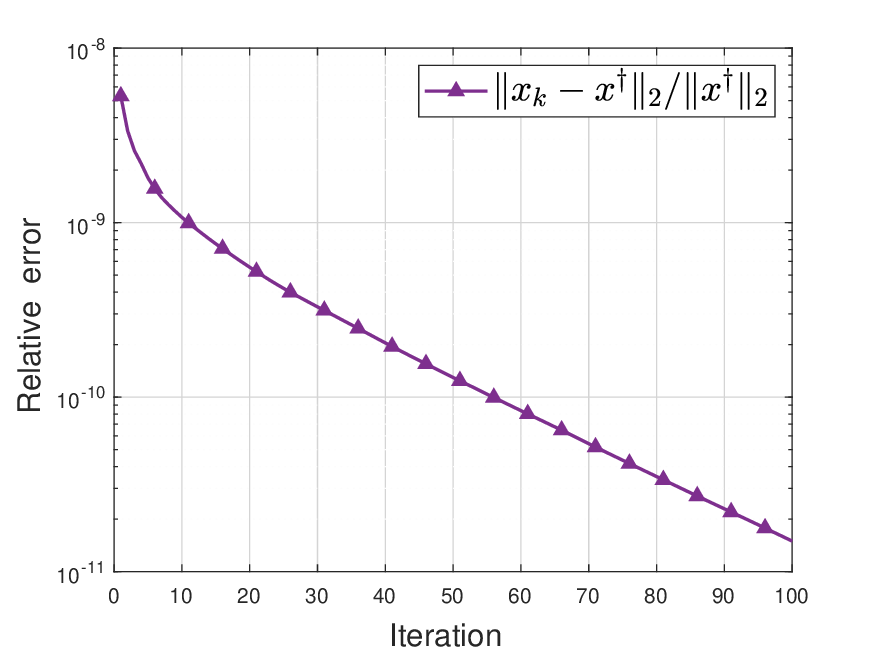}}\hspace{-2.5mm}
	\subfloat[]
	{\label{fig:3b}\includegraphics[width=0.45\textwidth]{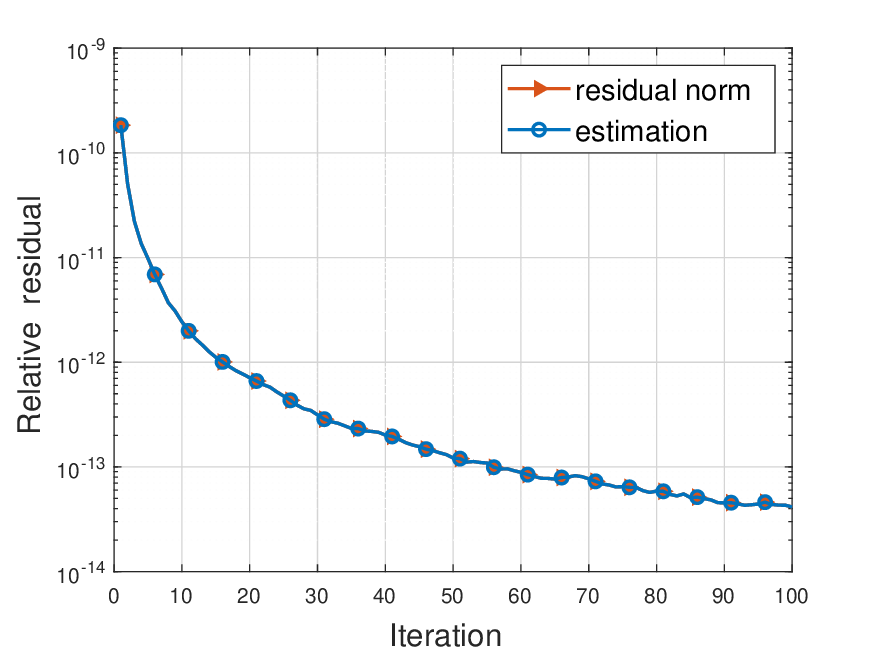}} \vspace{-0.0mm}
	\vspace{-2mm}
	\caption{Convergence history of \textsf{gLSQR} for the GLS problem with matrices \{{\sf TF15}, $L_2$\}: (a) relative error of iterative solutions; (b) directly computed relative residual norm and its estimation.}
	\label{fig3}
\end{figure}

\begin{figure}[!htbp]
	\centering
	\subfloat
	{\label{fig:4}\includegraphics[width=0.45\textwidth]{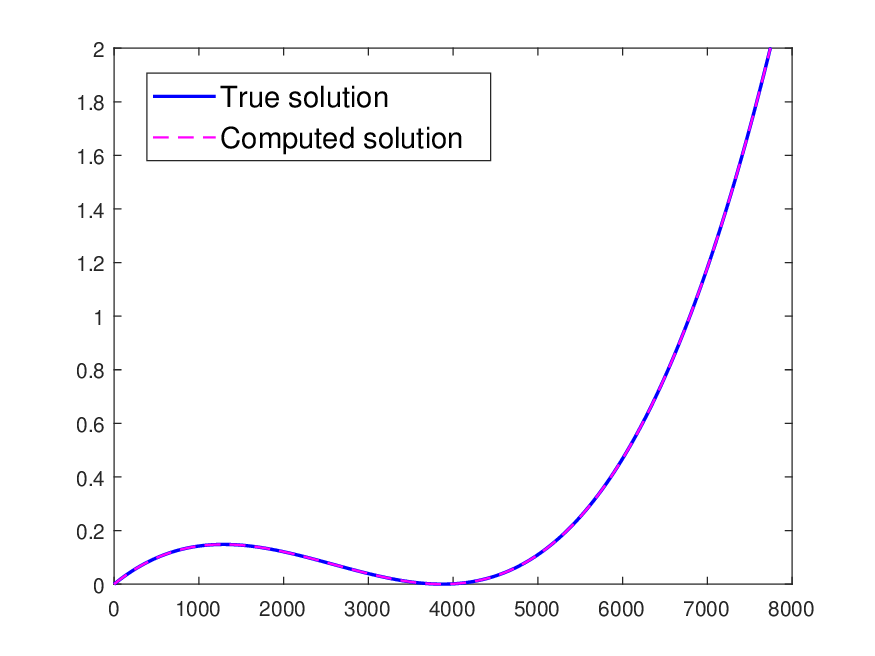}}\hspace{-0.0mm}
	\vspace{-3mm}
	\caption{True and computed solutions of the GLS problem with matrices \{{\sf TF15}, $L_2$\}.}
	\label{fig4}
\end{figure}

The computational results obtained with \textsf{gLSQR} are presented in \Cref{fig3} and \Cref{fig4}, which are very similar to the first experiment. The results demonstrate the effectiveness of \textsf{gLSQR} in iteratively solving the GLS problem.

\paragraph*{Experiment 3}
The matrix $A\in\mathbb{R}^{3700\times 8291}$, named {\sf ch} and originating from linear programming problems, is taken from \cite{Davis2011}.  Here, $L$ is set as $L=L_{1}$. In this set, $G$ is positive definite. We construct $w\in\mathbb{R}^{n}$ by evaluating the function $f(t)=\sin(5t)-2\cos(t)$ on a uniform grid over the interval $[-\pi,\pi]$. 

\begin{figure}[!htbp]
	\centering
	\subfloat[]
	{\label{fig:5a}\includegraphics[width=0.45\textwidth]{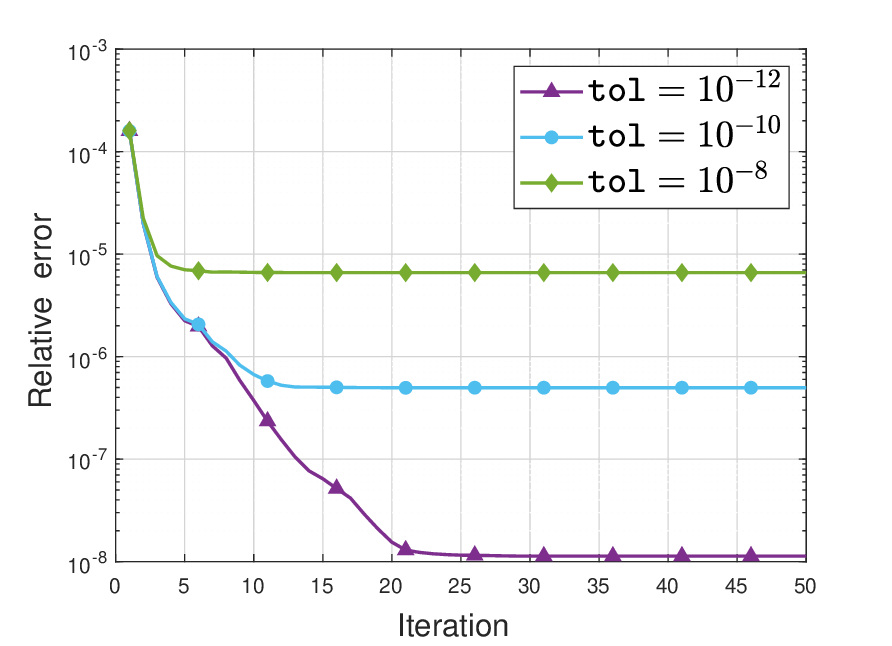}}\hspace{-2.5mm}
	\subfloat[]
	{\label{fig:5b}\includegraphics[width=0.43\textwidth]{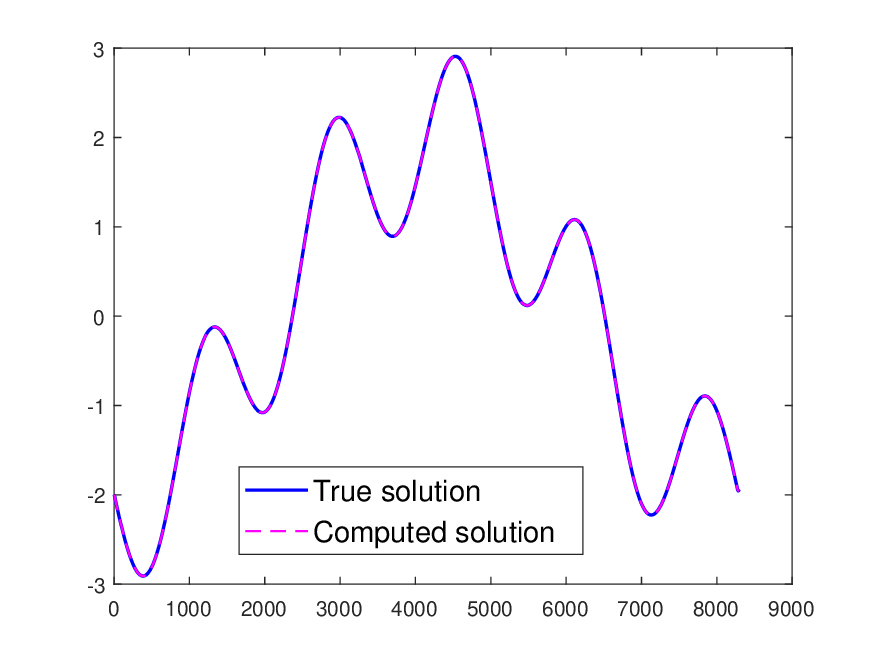}} \vspace{-0.0mm}
	\vspace{-2mm}
	\caption{Computed results of the GLS problem with matrices \{{\sf ch}, $L_1$\} by \textsf{gLSQR}, where at each iteration we use LSQR to approximate $s=G^{\dag}\bar{s}$ by solving $\min_{s}\|Gs-\bar{s}\|_2$ with different stopping tolerance value $\tau$: (a) relative error of iterative solutions with different $\tau$; (b) true and computed solutions with $\tau=10^{-8}$.}
	\label{fig5}
\end{figure}

We use this example to examine how the inaccurate computation of $s=G^{\dag}\bar{s}$ affects the numerical behavior of \textsf{gLSQR}. At each iteration, we approximately compute $s=G^{\dag}\bar{s}$ using MATLAB's built-in function \texttt{lsqr.m} to iteratively solve $\min_{s}\|Gs-\bar{s}\|_2$. The stopping tolerance value $\tau$ for \texttt{lsqr.m} is set to three different values. From \Cref{fig:5a}, we observe that the value of $\tau$ significantly impacts the final accuracy of $x_k$, with the accuracy being approximately on the order of $\mathcal{O}(\tau)$. We suspect that the final accuracy may be influenced by both the value of $\tau$ and the condition number of $G$. This should be explored further in future work.

Although constructing a very large-scale test example is difficult, we note that the main computational bottleneck of \textsf{gLSQR} is the computation of $G^{\dag}\bar{s}$. If a sparse Cholesky factorization of $G$ is available, it can be computed using a direct solver. Otherwise, an iterative solver such as LSQR or conjugate gradient (CG) is the only option. In this case, employing an effective preconditioner is crucial to accelerate the convergence for solving $\min_{s}\|Gs-\bar{s}\|_2$. Future work will involve constructing more large-scale test problems to evaluate the algorithm and exploring additional theoretical and computational aspects to enhance the performance of \textsf{gLSQR}.

\section{Conclusion}\label{sec6}
In this paper, we provide a new interpretation of the weighted pseudoinverse arising from the GLS problem $\min\|Lx\|_2 \ \mathrm{s.t.} \ \|M(Ax-b)\|_2=\min$. By introducing the linear operator $\calA:\calX= (\calR(G),\langle\cdot,\cdot\rangle_{G}) \rightarrow (\calR(P),\langle\cdot,\cdot\rangle_{P}), \ v \mapsto \calP_{\calR(P)}Av$ with $P=M^{\top}M$ and $G=A^{\top}PA+L^{\top}L$, we have shown that the minimum 2-norm solution of the GLS problem is the minimum $\calX$-norm solution of the LS problem $\min_{x\in\calX}\|\calA x-\calP_{\calR(P)}b\|_{P}$. Consequently, the weighted pseudoinverse $A_{ML}^{\dag}$ is shown to be equivalent to $\calA^{\dag}\calP_{\calR(P)}$ under the canonical bases. With this new interpretation of $A_{ML}^{\dag}$, we have derived a set of generalized Moore-Penrose equations that completely characterize the weighted pseudoinverse, and provided a closed-form expression for $A_{IL}^{\dag}$ using the GSVD of $\{A,L\}$. We have generalized the GKB process and proposed the \textsf{gLSQR} algorithm for iteratively computing $A_{ML}^{\dag}b$ for any given $b\in\mathbb{R}^{m}$, which allows us to compute the minimum 2-norm solution of the GLS problem. Several numerical examples have been used to test the \textsf{gLSQR} algorithm, demonstrating that it can compute the solution of a GLS problem with satisfactory accuracy. 

The results in this paper suggest that the closely related concepts---GLS, weighted pseudoinverse, GSVD, gGKB, and gLSQR---are appropriate generalizations of the classical concepts LS, pseudoinverse, SVD, GKB, and LSQR. This provides new tools for both analysis and computation of related applications.


\bibliographystyle{siamplain}
\bibliography{references}

\begin{thebibliography}{10}

\bibitem{ben2006generalized}
{\sc A.~Ben-Israel and T.~N. Greville}, {\em Generalized inverses: theory and
  applications}, Springer Science \& Business Media, 2006.

\bibitem{callon1987method}
{\sc G.~Callon and C.~Groetsch}, {\em The method of weighting and approximation
  of restricted pseudosolutions}, Journal of approximation theory, 51 (1987),
  pp.~11--18.

\bibitem{campbell2009generalized}
{\sc S.~L. Campbell and C.~D. Meyer}, {\em Generalized inverses of linear
  transformations}, SIAM, 2009.

\bibitem{caruso2019convergence}
{\sc N.~A. Caruso and P.~Novati}, {\em Convergence analysis of {LSQR} for
  compact operator equations}, Linear Algebra Appl., 583 (2019), pp.~146--164.

\bibitem{chipman1964least}
{\sc J.~S. Chipman}, {\em On least squares with insufficient observations},
  Journal of the American Statistical Association, 59 (1964), pp.~1078--1111.

\bibitem{cline1970extension}
{\sc R.~E. Cline and T.~N.~E. Greville}, {\em An extension of the generalized
  inverse of a matrix}, SIAM Journal on Applied Mathematics, 19 (1970),
  pp.~682--688.

\bibitem{Davis2011}
{\sc T.~A. Davis and Y.~Hu}, {\em The university of florida sparse matrix
  collection}, ACM Trans. Math. Software (TOMS), 38 (2011), pp.~1--25,
  \url{http://www.cise.ufl.edu/research/sparse/matrices}.

\bibitem{elden1982weighted}
{\sc L.~Eld{\'e}n}, {\em A weighted pseudoinverse, generalized singular values,
  and constrained least squares problems}, BIT Numerical Mathematics, 22
  (1982), pp.~487--502.

\bibitem{engl1996regularization}
{\sc H.~W. Engl, M.~Hanke, and A.~Neubauer}, {\em Regularization of inverse
  problems}, vol.~375, Springer Science \& Business Media, 1996.

\bibitem{golub1965calculating}
{\sc G.~Golub and W.~Kahan}, {\em Calculating the singular values and
  pseudo-inverse of a matrix}, Journal of the Society for Industrial and
  Applied Mathematics, Series B: Numerical Analysis, 2 (1965), pp.~205--224.

\bibitem{Golub2013}
{\sc G.~H. Golub and C.~F. Van~Loan}, {\em Matrix Computations}, The Johns
  Hopkins University Press, Baltimore, 4th~ed., 2013.

\bibitem{groetsch1977generalized}
{\sc C.~W. Groetsch}, {\em Generalized inverses of linear operators:
  representation and approximation}, (No Title),  (1977).

\bibitem{hansen1995lanczos}
{\sc P.~Hansen and M.~Hanke}, {\em A {L}anczos algorithm for computing the
  largest quotient singular values in regularization problems}, in SVD and
  Signal Processing III, Elsevier, 1995, pp.~131--138.

\bibitem{Hansen1998}
{\sc P.~C. Hansen}, {\em Rank-deficient and {D}iscrete {I}ll-{P}osed
  {P}roblems: {N}umerical {A}spects of {L}inear {I}nversion}, SIAM,
  Philadelphia, 1998.

\bibitem{Hansen2010}
{\sc P.~C. Hansen}, {\em Discrete {I}nverse {P}roblems: {I}nsight and
  {A}lgorithms}, SIAM, Philadelphia, 2010.

\bibitem{hartung1979note}
{\sc J.~Hartung}, {\em A note on restricted pseudoinverses}, SIAM Journal on
  Mathematical Analysis, 10 (1979), pp.~266--273.

\bibitem{Ramsay2009}
{\sc S.~G. James~Ramsay, Giles~Hooker}, {\em Functional Data Analysis with {R}
  and {MATLAB}}, Springer New York, NY, 2000.

\bibitem{jia2023joint}
{\sc Z.~Jia and H.~Li}, {\em The joint bidiagonalization method for large gsvd
  computations in finite precision}, SIAM Journal on Matrix Analysis and
  Applications, 44 (2023), pp.~382--407.

\bibitem{Kress2014}
{\sc R.~Kress}, {\em Linear Integral Equations}, Springer New York, NY, 2013,
  \url{https://doi.org/10.1007/978-1-4614-9593-2}.

\bibitem{li2024characterizing}
{\sc H.~Li}, {\em Characterizing {GSVD} by singular value expansion of linear
  operators and its computation}, arXiv:2404.00655,  (2024).

\bibitem{marquardt1970generalized}
{\sc D.~W. Marquardt}, {\em Generalized inverses, ridge regression, biased
  linear estimation, and nonlinear estimation}, Technometrics, 12 (1970),
  pp.~591--612.

\bibitem{milne1968oblique}
{\sc R.~Milne}, {\em An oblique matrix pseudoinverse}, SIAM Journal on Applied
  Mathematics, 16 (1968), pp.~931--944.

\bibitem{minamide1970restricted}
{\sc N.~Minamide and K.~Nakamura}, {\em A restricted pseudoinverse and its
  application to constrained minima}, SIAM Journal on Applied Mathematics, 19
  (1970), pp.~167--177.

\bibitem{mitra1974projections}
{\sc S.~K. Mitra and C.~R. Rao}, {\em Projections under seminorms and
  generalized moore penrose inverses}, Linear Algebra and its applications, 9
  (1974), pp.~155--167.

\bibitem{moore1920reciprocal}
{\sc E.~H. Moore}, {\em On the reciprocal of the general algebraic matrix},
  Bulletin of the american mathematical society, 26 (1920), pp.~294--295.

\bibitem{Moore1935-MOOGA}
{\sc E.~H. Moore}, {\em General Analysis}, The American philosophical society,
  Philadelphia,, 1935.

\bibitem{Paige1981}
{\sc C.~C. Paige and M.~A. Saunders}, {\em Towards a generalized singular value
  decomposition}, SIAM J. Numer. Anal., 18 (1981), pp.~398--405.

\bibitem{Paige1982}
{\sc C.~C. Paige and M.~A. Saunders}, {\em {LSQR}: An algorithm for sparse
  linear equations and sparse least squares}, ACM Trans. Math. Software, 8
  (1982), pp.~43--71.

\bibitem{penrose1955generalized}
{\sc R.~Penrose}, {\em A generalized inverse for matrices}, in Mathematical
  proceedings of the Cambridge philosophical society, vol.~51, Cambridge
  University Press, 1955, pp.~406--413.

\bibitem{penrose1956best}
{\sc R.~Penrose}, {\em On best approximate solutions of linear matrix
  equations}, in Mathematical Proceedings of the Cambridge Philosophical
  Society, vol.~52, Cambridge University Press, 1956, pp.~17--19.

\bibitem{price1964matrix}
{\sc C.~M. Price}, {\em The matrix pseudoinverse and minimal variance
  estimates}, SIAM review, 6 (1964), pp.~115--120.

\bibitem{ward1971weighted}
{\sc J.~Ward, T.~Boullion, and T.~Lewis}, {\em Weighted pseudoinverses with
  singular weights}, SIAM Journal on Applied Mathematics, 21 (1971),
  pp.~480--482.

\bibitem{ward1977limit}
{\sc J.~F. Ward, Jr}, {\em On a limit formula for weighted pseudoinverses},
  SIAM Journal on Applied Mathematics, 33 (1977), pp.~34--38.

\bibitem{weiss2002advanced}
{\sc V.~Weiss, L.~Andor, G.~Renner, and T.~V{\'a}rady}, {\em Advanced surface
  fitting techniques}, Computer Aided Geometric Design, 19 (2002), pp.~19--42.

\bibitem{wendland2004scattered}
{\sc H.~Wendland}, {\em Scattered data approximation}, vol.~17, Cambridge
  {U}niversity {P}ress, 2004.

\bibitem{Zha1996}
{\sc H.~Zha}, {\em Computing the generalized singular values/vectors of large
  sparse or structured matrix pairs}, Numer. Math., 72 (1996), pp.~391--417.

\end{thebibliography}

\end{document}